\documentclass[11pt,letter]{amsart}
\usepackage{amscd}
\usepackage{amssymb,amsmath}
\usepackage[centering,text={15.5cm,23cm}]{geometry} 
\usepackage{graphicx,color}
\usepackage{hyperref}
\hypersetup{
  colorlinks   = true,          
  urlcolor     = violet,          
  linkcolor    = blue,          
  citecolor   = violet             
}
\usepackage{comment}
\usepackage{wrapfig}
\usepackage[all]{xy}
\usepackage{mathrsfs}
\usepackage{mathtools}
\usepackage{marvosym}
\usepackage{stmaryrd}
\usepackage{bm}
\usepackage{srcltx}
\usepackage{upgreek} 
\usepackage[inline]{enumitem} 
\usepackage[margin=4ex]{caption} 
\usepackage[framemethod=TikZ]{mdframed}
\mdfdefinestyle{MyFrame}{%
    linecolor=blue,
    outerlinewidth=0.7pt,
    roundcorner=12pt,
    innertopmargin=.5\baselineskip,
    innerbottommargin=.5\baselineskip,
    innerrightmargin=1pt,
    innerleftmargin=1pt,
    backgroundcolor=white}

\usepackage{multicol}
\usepackage{rotating}

\definecolor{shadecolor}{rgb}{1,0.9,0.7}

\newtheorem{theorem}{Theorem}[section]
\newtheorem{lemma}[theorem]{Lemma}
\newtheorem{lemma-definition}[theorem]{Lemma-Definition}
\newtheorem{proposition}[theorem]{Proposition}
\newtheorem{corollary}[theorem]{Corollary}

\theoremstyle{definition}

\newtheorem{remark}[theorem]{Remark}

\numberwithin{equation}{section}
\numberwithin{figure}{section}

\newcommand {\lfor} {\llbracket}
\newcommand {\rfor} {\rrbracket}

\newcommand{\NN} {\mathbb{N}}
\newcommand{\ZZ} {\mathbb{Z}}
\newcommand{\QQ} {\mathbb{Q}}
\newcommand{\RR} {\mathbb{R}}
\newcommand{\CC} {\mathbb{C}}

\newcommand{\PP} {\mathbb{P}}
\renewcommand{\AA} {\mathbb{A}}
\newcommand{\GG} {\mathbb{G}}

\newcommand {\shC}  {\mathcal{C}}
\newcommand {\shD}  {\mathcal{D}}

\newcommand {\shE}  {\mathcal{E}}

\newcommand {\shO}  {\mathcal{O}}

\newcommand {\shR}  {\mathcal{R}}

\newcommand {\shX}  {\mathcal{X}}
\newcommand {\shY}  {\mathcal{Y}}

\newcommand {\foX}  {\mathfrak{X}}

\newcommand {\fob}  {\mathfrak{b}}

\newcommand {\fop}  {\mathfrak{p}}

\newcommand {\fot}  {\mathfrak{t}}
\newcommand {\fou}  {\mathfrak{u}}

\newcommand {\Aff}  {\operatorname{Aff}}
\newcommand {\an}  {\mathrm{an}}

\newcommand {\const} {{\mathrm{const}}}

\newcommand {\dlog} {\operatorname{dlog}}

\newcommand {\gp}  {{\operatorname{gp}}}

\newcommand {\hol}  {\mathrm{hol}}
\newcommand {\Hom}  {\operatorname{Hom}}

\newcommand {\id}  {\operatorname{id}}

\newcommand {\Int}  {\operatorname{Int}}

\newcommand {\liminv} {\varprojlim}

\newcommand {\lra}  {\longrightarrow}

\newcommand {\M} {\mathcal{M}}

\newcommand {\maxid} {\mathfrak{m}}

\newcommand {\NE}  {\operatorname{NE}}

\renewcommand{\O}  {\mathcal{O}}

\newcommand {\ol} {\overline}
\newcommand {\ord}  {\operatorname{ord}}

\newcommand {\out}  {\mathrm{out}}

\renewcommand{\P}  {\mathscr{P}}

\newcommand {\PGL}  {\operatorname{PGL}}
\newcommand {\PL} {\operatorname{PL}}

\newcommand {\Proj} {\operatorname{Proj}}

\newcommand {\scrM}  {\mathscr{M}}

\newcommand {\Spec} {\operatorname{Spec}}

\newcommand {\Spf}  {\operatorname{Spf}}

\newcommand {\trop}  {{\operatorname{trop}}}

\newcommand {\ul} {\underline}

\newcommand{\ptwo}{\mathbb{P}^2}

\def\mydate{\ifcase\month \or January\or February\or March\or
April\or May\or June\or July\or August\or September\or October\or
November\or December\fi \space\number\day,\space\number\year}

\usepackage{graphicx}
\usepackage{pgf,tikz}\usetikzlibrary{arrows.meta,calc}
\usepackage{pgfplots}
\usepgfplotslibrary{colormaps}
\pgfplotsset{
	compat=newest,
	colormap={mycolormap}{color=(lightgray) color=(white) color=(lightgray)}
}
\usetikzlibrary{decorations.pathmorphing}
\tikzset{snake it/.style={decorate, decoration=snake}}
\usepackage{tikz-3dplot}

\newcommand{\hhh}{{\rm H}}
\newcommand{\ecircq}{\check{E}^\circ_Q}
\newcommand{\diff}{{\rm d}}
\newcommand{\myclr}{magenta!60}

\DeclareFontFamily{U}{mathx}{}
\DeclareFontShape{U}{mathx}{m}{n}{ <-> mathx10 }{}
\DeclareSymbolFont{mathx}{U}{mathx}{m}{n}
\DeclareFontSubstitution{U}{mathx}{m}{n}
\DeclareMathAccent{\widecheck}{0}{mathx}{"71}

\date{\today}

\begin{document}

\title[Intrinsic enumerative mirror symmetry]
{Intrinsic enumerative mirror symmetry: Takahashi's log mirror symmetry for $(\PP^2,E)$ revisited}
\author{Michel van Garrel, Helge Ruddat, Bernd Siebert}

\address{\tiny School of Mathematics, University of Birmingham,
Birmingham B15 2TT, UK}
\email{m.vangarrel@bham.ac.uk}

\address{\tiny
Department of Mathematics and Physics, Univ. Stavanger, P.O. Box 8600 Forus, 4036 Stavanger, Norway}
\email{helge.ruddat@uis.no}

\address{\tiny Department of Mathematics, University of Texas at Austin,
2515 Speedway, Stop C1200, Austin, TX 78712, USA}
\email{siebert@math.utexas.edu}
\thanks{}

\begin{abstract}
Let $E$ be a smooth cubic in the projective plane $\ptwo$. Nobuyoshi Takahashi
formulated a conjecture that expresses counts of rational curves of varying
degree in $\ptwo\setminus E$ as the Taylor coefficients of a particular period
integral of a pencil of affine plane cubics after reparametrizing the pencil
using the exponential of a second period integral.

The intrinsic mirror construction introduced by Mark Gross and the third author
associates to a degeneration of $(\ptwo, E)$ a canonical wall structure from
which one constructs a family of projective plane cubics that is birational to
Takahashi's pencil in its reparametrized form. By computing the period integral
of the positive real locus explicitly, we find that it equals the logarithm of
the product of all asymptotic wall functions. The coefficients of these
asymptotic wall functions are logarithmic Gromov-Witten counts of the central
fiber of the degeneration that agree with the algebraic curve counts in
$(\PP^2,E)$ in question. We conclude that Takahashi's conjecture is a natural
consequence of intrinsic mirror symmetry. Our method generalizes to give similar
results for log Calabi-Yau varieties of arbitrary dimension.
\end{abstract}

\dedicatory{This article is dedicated to Gang Tian, who introduced the third
author to quantum cohomology in 1993 and to so many other things, at the
occasion of his 65th birthday.}

\maketitle
\setcounter{tocdepth}{2}
\tableofcontents



\section{Enumerative mirror symmetry is manifest in intrinsic mirror symmetry}


\subsection{Intrinsic enumerative mirror symmetry}
Mirror symmetry is a prediction originating in string theory
\cite{CLS,GrPl,COGP}. It states that Calabi--Yau manifolds come in \emph{mirror
families}
\[
\mathcal{X}_{\tau,s} \longleftrightarrow \widecheck{\mathcal{X}}_{\sigma,t}
\]
with $\tau,\sigma$ $A$-model symplectic structure parameters
and $s,t$ $B$-model complex structure parameters.

The construction of pairs of mirror families traditionally mostly relied on
toric ad hoc constructions \cite{B,BB,BH}, but can now be done in great
generality in a completely natural fashion via the intrinsic mirror symmetry
construction \cite{GHK,GS18,KY23,GS19,GS22,KY24}.

\emph{Enumerative mirror symmetry}, pioneered in \cite{COGP}, is the conjecture
that the $A$-model variation of symplectic structure of $\mathcal{X}_{\tau,s}$
around a large volume limit point $\tau_0$ is identified with the $B$-model
variation of complex structure of $\widecheck{\mathcal{X}}_{\sigma,t}$ around
the mirror large complex structure limit point $t_0$, possibly after a canonical
identification (\emph{mirror map})
\[
\tau \longleftrightarrow t \,.
\]
A distinguished solution of the $A$-model variation problem is the $A$-model
potential of $\mathcal{X}_{\tau,s}$, a generating function of Gromov--Witten
invariants in the variable $e^\tau$. The analogue in the $B$-model is a
potential function for the Yukawa coupling encoding variations of the Hodge
structure.

We prove a version of the enumerative mirror correspondence from the intrinsic
mirror perspective in the case of the affine Calabi--Yau $\ptwo\setminus E$ for
$E$ a smooth cubic, or rather for the pair $(\PP^2,E)$. One crucial ingredient
in our proof is that the coordinate ring of the intrinsic mirror of a space $Y$
is based on Gromov-Witten invariants of $Y$. It is therefore completely natural
to expect that the periods of the mirror somehow contain enumerative
informations of $Y$. The difficulty of course is how exactly the curve counting
invariants in the definition of the intrinsic mirror influence integrals over
the holomorphic volume form.

One crucial ingredient for this analysis is that the intrinsic mirror
$\widecheck{\mathcal{X}}_{\sigma,t}$ has a positive real locus. The integral of
the holomorphic volume form over this locus readily gives the $B$-model
potential function and can be easily computed in the present case. See also
\cite{AGIS,I1} for some general conjectures that the period integral over the
positive real locus agrees with the Givental $J$-function, a refined $A$-model
potential function, up to terms coming from the Gamma class and after applying
the mirror map.

Another important fact is that intrinsic mirror symmetry is naturally
parametrized in canonical coordinates, so the mirror map is trivial. This was
shown for the case of compact Calabi-Yau manifolds in \cite{RS}. The present
case is treated by a trivial generalization of \cite{RS} from families of
singular cycles to families of singular chains, see \S\ref{Subsect: Canonical
coordinates}. Triviality of the mirror map allows to identify contributions to
the period integral from each individual enumerative invariant.

Taken together, we have shown that intrinsic mirror symmetry gives a transparent
and almost trivial proof of log enumerative mirror symmetry as conjectured by
N.\ Takahashi in \cite{T01}, a highly non-trivial statement. More general
results will appear in future work.

Another interpretation of our results is as the explicit computation of a normal
function associated to the family of fibers in the (fiberwise compactified)
mirror superpotential, an elliptic fibration.


\subsection{Outline and main result}

We proceed as follows:

\subsubsection*{$A$-model and Takahashi's mirror conjecture (\S\ref{sec:tak})}

We count $\AA^1$-curves in $\PP^2\setminus E$ and review Takahashi's mirror
conjecture for these counts. On a technical level, we define $N_d$ as the
maximal tangency genus 0 degree $d\ge 1$ log Gromov--Witten invariant of
$(\ptwo,E)$ counting stable log maps meeting $E$ in a single unspecified point
of tangency $3d$. E.g., $N_1=9$ is the number of lines tangent to one of the
flex points. There are several ways of direct computation of the $N_d$ without
recourse to a mirror family, notably \cite[Expl.2.2]{Ga} or Givental mirror
symmetry in the relative setting \cite{FTY,You}.

\subsubsection*{Degeneration (\S\ref{Subsect: Maximal degeneration})}

The input for the intrinsic mirror construction (relative case) is a log smooth
degeneration $g:\mathcal{Y}\to\AA^1$ of the pair $(\ptwo,E)$ such that the log
scheme $(\shY,\shD)$ has dimension 0 strata. Our model is an explicit family of
hypersurfaces in a weighted projective space.

For uniqueness, the classical theory of plane cubics shows that the moduli stack
$\scrM$ of pairs $(Y,D)$ isomorphic to $\PP^2$ with a smooth plane cubic $E$ is
one-dimensional. Indeed, the $j$-invariant of $E$ provides a finite-to-one map
to $\AA^1$. Our degeneration is a partial compactification of the universal
family over $\scrM$ near the unique maximal degeneration point of $\scrM$, see
Remark~\ref{Rem: Uniqueness of shY}.

\subsubsection*{Intrinsic mirror family (\S\ref{Subsect:
Tropicalization}--\ref{Subsect: Geometry of intrinsic mirror})}

Applying \cite{GS18,GS19,GS22} associates to $g:\shY\to\AA^1$ its intrinsic
mirror family $\shX \to \Spec \CC\lfor t\rfor$. By \cite{CPS}, there is a
sub-family $\shX \supset w^{-1}(1) \to \Spec \CC\lfor t\rfor$ which is the
intrinsic mirror family to the embedded $\mathcal E\subset \mathcal Y$, the
horizontal part of $\shD=\shE\cup \shY_0$; moreover, this mirror family differs
by an explicit base-change from the intrinsic mirror family to $\mathcal E$. The
family turns out to be the formal completion at $t=0$ of a flat analytic family
$\shX_\an\to \CC$ (\S\ref{Subsect: Analytification}). By abuse of notation, we
write $\shX_t$ for the fiber over $t\in\CC$.

\subsubsection*{Analytification and positive real locus (\S\ref{Subsect:
Analytification} and \ref{Subsect: positive real locus})}

Working in the complex-analytic category and restricting to real $t>0$, the
intrinsic mirror family admits a positive real sub-locus $\shX_t^>$ obtained by
restricting the coordinates to lie in $\RR_{>0}$. Topologically,
$\shX_t^>\simeq\RR^2$ and $\shX_t^>\cap w^{-1}(1)$ is a circle. Denote by
$\Gamma_t\subset \shX^>_t$ the Lefschetz thimble, topologically a closed disc,
with $\partial \Gamma_t = \shX_t^>\cap w^{-1}(1)$. Note this Lefschetz thimble
extends to any contractible subset of a sufficiently small pointed disc
$0<|t|<\varepsilon$, but exhibits monodromy about $t=0$. We compute the period
of $\Gamma_t$ over the normalized holomorphic volume form $\Omega_t$ on $\shX_t$
with $t>0$ real and then extend by holomorphic continuation.

\subsubsection*{Canonical coordinates (\S\ref{Subsect: Canonical coordinates})}

Another family of $2$-chains $\beta_t$ with boundary on $w^{-1}(1)\cap\shX_t$ is
of the form constructed in \cite{RS}, and $\int_{\beta_t}\Omega_t$ yields $\log
t$ up to a constant. The intrinsic mirror to the embedding $\mathcal E
\hookrightarrow \mathcal Y$ is the mirror family $w^{-1}(1) \to \Spec \CC\lfor
t\rfor$. Thus the period integral $\int_{\Gamma_t}\Omega_t$ is already in the
correct coordinate $t$.

\subsubsection*{Integral over positive real Lefschetz thimbles
(\S\ref{sec:main-result})}

The main technical result is a computation of $\Omega_t$ over a difference of
real Lefschetz thimbles with boundaries on $w^{-1}(s_0)$ and $w^{-1}(s_1)$
(Proposition~\ref{prop:period}) with a transparent enumerative meaning. We then
use a $\CC^*$-action on the mirror family acting with weight one on both $t$ and
$s$ to compute $\int_{\Gamma_t}\Omega_t$. The same reasoning applies in more
general cases than $(\ptwo,E)$.

\subsubsection*{Main result: Equality of $A$-model and $B$-model potentials (\S\ref{sec:main-result})}

Putting the two period computations together gives the following main theorem.

\begin{theorem}
\label{thm:main}
Intrinsic enumerative mirror symmetry holds for the log smooth degeneration
$g:\shY \to\AA^1$ of $(\ptwo,E)$, namely
\begin{equation}
\label{eq:period-int}
\int_{\Gamma_t}\Omega_t = \frac{1}{2}\cdot\log^2 (t^3) + c +
\sum_{d=1}^\infty N_d \, t^{3d},
\end{equation}
for some constant $c\in \CC$.
\end{theorem}

\begin{remark}
The constant $c$ in \eqref{eq:period-int} can easily be computed as in
\cite{AGIS} to be $c=-3\zeta(2)$.
\end{remark}

As a corollary, we deduce Takahashi's enumerative mirror conjecture in
\S\ref{subsec:match}. We emphasize that the proof of Theorem \ref{thm:main} is
by direct and rather trivial computation. Moreover, the matching of symplectic
and Kähler parameters is already built into intrinsic mirror symmetry.

\subsubsection*{Acknowledgment}

{\small
We thank Pierrick Bousseau, Tim Gräfnitz, Nobuyoshi Takahashi, Honglu
Fan and Fenglong You for discussions
related to their respective works. M.\,van
Garrel would like to thank the Isaac Newton Institute for Mathematical Sciences,
Cambridge, for support and hospitality during the programme K-theory, algebraic
cycles and motivic homotopy theory where work on this paper was undertaken. This
work was supported by EPSRC grant no EP/R014604/1. H.\,Ruddat is supported by
the NFR Fripro grant Shape2030. Research by B.\,Siebert was partially supported
by NSF grants DMS-1903437 and DMS-2401174.
}


\section{Takahashi's log mirror symmetry conjecture for $(\mathbb{P}^2,E)$}
\label{sec:tak}

A version of Theorem \ref{thm:main} was conjectured in the work of N.\ Takahashi \cite{T01}, which we review now.


\subsection{$A$-model}
\label{sec:A-model}

Let $\overline{M}_{0,(3d)}(\ptwo,E,d)$ be the moduli space of genus 0 degree $d$
basic stable log maps to $(\ptwo,E)$ meeting $E$ in one unspecified point of
maximal tangency $3d$ \cite{Ch,AC,GS13}. It is of virtual dimension $0$ and
admits a virtual fundamental class $[\overline{M}_{0,(3d)}(\ptwo,E,d)]^{\rm
vir}\in\hhh_{0}\left(\overline{M}_{0,(3d)}(\ptwo,E,d),\QQ\right)$. We obtain the
degree $d$ maximal tangency genus 0 log Gromov--Witten invariants
\[
N_d:=\int_{[\overline{M}_{0,(3d)}(\ptwo,E,d)]^{\rm vir}}1 \in\QQ
\]
virtually counting rational curves in $\ptwo$ that intersect $E$ in exactly one
point.


\subsection{Takahashi's mirror family}
\label{Subsect: Takahashi's mirror family}
We review the setup from \cite{T01}. Takahashi constructs the mirror family from
toric duality principles. The fan of $\ptwo$ is generated by
$(1,0),\,(0,1),\,(-1,-1)\in \ZZ^2\otimes_\ZZ\RR$. Let $\Delta^\circ$ be the
reflexive polygon with vertices $(1,0),\,(0,1),\,(-1,-1)$. Then the associated
toric variety is $\mathbb P_{\Delta^\circ}\cong\ptwo\big/\mu_3$, where
$\lambda\in\mu_3=\{1,e^{\pm2\pi i/3}\}$ acts on homogeneous coordinates by
\[
X\mapsto X, \quad Y \mapsto \lambda \cdot Y, \quad Z \mapsto \lambda^2 \cdot Z,
\]
and the $\mu_3$-invariant sections $X^3,\,Y^3,\,Z^3,XYZ$ form a basis of
$|-K_{\mathbb P_{\Delta^{\circ}}}|$. On the complement $(\CC^*)^2$ of the
coordinate lines in $\PP^2/\mu_3$, the sections restrict to the Laurent
monomials $x,\, y,\, \frac{1}{xy}$ and $1$. According to \cite{B,BB}, the
mirror-dual to $E$ is the anticanonical pencil
\begin{equation}
\label{eq:takfam}
\check{E}_t \, : \, XYZ-t\left(X^3+Y^3+Z^3\right)=0
\end{equation}
whose members are smooth if $t\in \mathbb C\setminus\left(\left\{0\right\}\cup
\frac{1}{3} \, \mu_3\right)$. For $t\in\frac{1}{3}\mu_3$, $\check{E}_t$ is a
singular elliptic curve with a single node whereas $\check{E}_0$ has three
nodes. Takahashi considers the restriction of the pencil \eqref{eq:takfam} to
the torus $(\mathbb C^*)^2\subset\ptwo/\mu_3$,
\begin{equation}
\label{eq:ecirchat}
\check{E}^\circ_t \, : \, t \left( x+y+\frac{1}{xy} \right) = 1\,,
\quad x=\frac{X^3}{XYZ},\ y=\frac{Y^3}{XYZ}\,.
\end{equation}

Multiplying $t$ by a third root of unity leads to an isomorphic fiber. A natural parameter for the family of smooth elliptic curves is therefore
\begin{equation}
\label{eq:pQ}
Q=t^3\in \mathcal{M} := \Big(\mathbb C\setminus\big(\{0\}\cup
\textstyle\frac{1}{3} \, \mu_3\big)\Big)\Big/\mu_3,
\end{equation}
and $\mathcal{M}$ is the moduli space of elliptic curves with
$\Gamma_1(3)$-level structure.\footnote{A conic bundle over the family of
elliptic curves is the mirror to local $\ptwo$, the total space of
$\shO_{\ptwo}(-3)$ over $\ptwo$. This situation has been extensively studied,
see e.g.\ \cite{CKYZ,Gro,AKV,GZ,CLL,CLT,GS14,L}.} Its connection to the present
situation is established in \cite{B22,B23}.

Takahashi considers singular homology $2$-chains in $(\mathbb C^*)^2$ with
boundaries in $\check{E}^\circ_Q$. These chains can be chosen as families
of circles fibering over a path in the $Q$-plane. The path connects a fixed
regular value of $Q$ to the singular value $Q=1/27$, where the circle fibre
contracts to a point. One needs three such paths to obtain a basis of the
relative homology group $\hhh_2\big((\CC^*)^2,\ecircq;\,\ZZ\big)$. To obtain them,
one can descend the three linear paths from the $t$-plane which run from a fixed
choice of cube root of $Q$ to the three points $\frac{1}{3} \, \mu_3$. Projected
to the $Q$-plane, these paths are homotopic up to closed loops with different
winding numbers around the origin. The corresponding relative cycles
$\Gamma^0_Q$, $\Gamma^1_Q$, $\Gamma^2_Q$ are Lefschetz thimbles for the function
$x+y+\frac{1}{xy}$. Their boundary circles in $\check{E}^\circ_Q$ are
$\gamma,T\gamma,T^2\gamma$ where $T$ denotes the monodromy endomorphism of
$\hhh_1\big(\check{E}^\circ_Q,\ZZ\big)$ for a simple loop about $Q=0$. The
identity $(T-\id)^2=0$ gives the generator $\Gamma^0_Q-2\Gamma^1_Q+\Gamma^2_Q$
of the linear relations among the three boundary circles.

Takahashi considers the period integrals
\[
I_{\Gamma^i_Q}(Q):=\int_{\Gamma^i_Q} \frac{\diff x}{x} \wedge \frac{\diff y}{y},
\]
which are multi-valued locally holomorphic functions in $Q\in\mathcal{M}$.
Writing $\theta:=Q\frac{\diff}{\diff Q}$, these periods satisfy the homogeneous
Picard-Fuchs equation\footnote{The same equation with opposite sign $Q\mapsto
-Q$ appears in the context of local mirror symmetry \cite{CKYZ}. On the $A$-side, the corresponding sign appears in the log-local correspondence \cite{vGGR}.}
\begin{equation}
\label{eq:pf}
\left\{\theta^3 - 3 Q \, \theta(3\theta + 1)(3\theta + 2)\right\}I = 0,
\end{equation}
which admits a basis of solutions of the form
\begin{align}
I_0(Q) &= 1, \nonumber \\
I_1(Q) &= \log Q + \sum_{d\geq 1} \, Q^d \, \frac{(3d)!}{d(d!)^3}\, ,
\label{eq:periods} \\
I_2(Q) &= - \frac{1}{2}\log^2 Q + I_1(Q) \log Q + I_2^{\rm hol}(Q)\,  \nonumber
\end{align}
for some unique single valued holomorphic function $I_2^{\rm hol}(Q)$. The
single valued first solution $I_0(Q)$ is a period integral obtained by
integrating $\frac1{(2\pi i)^2}\frac{\diff x}{x} \wedge \frac{\diff y}{y}$ over
a generator of $\hhh_2\big((\CC^*)^2,\ZZ\big)$. Expressing $I_1(Q)=\log Q +
I_1^{\rm hol}(Q)$, an elementary calculation shows that $I_1^{\rm hol}(Q)$ has a
convergence radius of $\frac1{27}$, consistent with the singular fibers'
location in the family.


\subsection{Enumerative Mirror Conjecture}
\label{Subsect: Enumerative mirror conjecture}
Conjecture~1.10 from \cite{T01} together with Remark~1.11 of the same source,
asserts that, using the canonical
coordinate $q$ defined by
\begin{equation}
\label{Eqn: canonical coordinate}
q := -e^{I_1(Q)},
\end{equation}
the following holds
\begin{equation}
\label{eq:takItwo}
I_2(q) = \frac{1}{2}\log^2(-q) + \sum_{d=1}^\infty N_d \, (-q)^d.
\end{equation}
Example 2.2 of \cite{Ga} proves \eqref{eq:takItwo} using recursion relations of
relative Gromov-Witten invariants. We give a completely different proof in this
article that performs an explicit period integral computation in the canonical
scattering diagram and explains the conjecture via the canonical wall structure
of intrinsic mirror symmetry. We will show in Proposition~\ref{prop:period1}
that $q=-t^3$ and therefore Theorem~\ref{thm:main} implies \eqref{eq:takItwo}.

The curious negative sign in $q$ was justified in \cite{T01} through the
monodromy computation of the system \eqref{eq:periods} about $Q=0$.
We will understand the sign from the odd valency of a tropical 1-cycle in an
explicit period integration from \cite{RS}, see Proposition~\ref{prop:period1}
and the proof of Proposition~\ref{Prop: Takahashi's conjecture via intrinsic
periods}.

\begin{remark}
The conjecture originated from explicit computations in \cite{T96} concerning
maps from the affine line into $(\mathbb{P}^2, E)$. Takahashi's full conjecture
predicts a further refinement of the sum in \eqref{eq:takItwo} that was proven
by Bousseau through an isomorphism of scattering diagrams \cite{B22,B23}, combined
with the tropical correspondence from \cite{Gr20}. The enumerative log-local
principle \cite{vGGR} connects the conjecture to the local mirror symmetry of
Calabi-Yau threefolds \cite{CKYZ}. Several works, including
\cite{B20,CvGKT,BBvG,vGNS,BS},
have explored aspects of this refined conjecture.
Its relation to the Givental formalism was studied in \cite{FTY,TY,You}.
\end{remark}


\section{Intrinsic mirror family}
\label{sec:degeneration}

We review the construction of the intrinsic mirror family $\shX \to \Spec
\CC\lfor t\rfor$ following \cite{GS18,GS19,GS22} and explain how it relates to
the construction considered in \cite{CPS,Gr20}.

The intrinsic mirror construction takes as input a simple normal crossings pair
$(Y,D)$ such that $K_Y+D$ is numerically equivalent to an effective
$\QQ$-divisor supported on $D$. In this paper we will only deal with the case
that $D$ is actually anti-canonical. In any case, the construction produces a
flat affine formal scheme $\foX=\Spf R$ over a completed version $\CC\lfor
P\rfor$ of the ring $\CC[\NE(Y)]$ generated by effective curve class on $Y$.
Here $P$ is isomorphic to the set of integral points of a strongly convex
rational polyhedral cone in $\RR^n$ together with a monoid homomorphism
$\NE(Y)\to P$. If $\NE(Y)$ is finitely generated, as in all cases considered in
this paper, one may take $P=\NE(Y)$, but we will ultimately base-change to
$P=\NN$ whose generator will give our degeneration parameter $t$. The reduced
fiber $X_0\subseteq \foX$ is isomorphic to $\Spec \operatorname{SR}(D)$ where
$\operatorname{SR}(D)$ is the Stanley-Reisner ring of the dual intersection
complex of $D$, a simplicial complex. Thus $X_0$ is a union of $\AA^k$'s glued
along intersections of coordinate hyperplanes, with one copy of $\AA^k$ for each
minimal stratum of $D$ of codimension $k$ in $Y$.

If $D$ supports a nef divisor, $R$ is the completion of a finitely generated $\CC \lfor
P\rfor$-algebra \cite[Constr.\,1.24]{GS19}, so one naturally obtains an algebraic family
\[
\shX=\Spec R\lra \Spec\CC \lfor P\rfor.
\]
If $D$ even supports an ample divisor, $R$ is of finite type over
$\CC[P]$ \cite[Rem.\,1.26]{GS19}. Thus in this case, the mirror family is given
by polynomial equations.

To obtain projective families of mirrors, one applies the previous construction
to a normal crossings degeneration $g:\shY\to S$ for some smooth curve $S$. The
intrinsic mirror ring $R$ then comes with a natural $\NN$-grading, and
$\shX=\Proj R$ gives the mirror family.


\subsection{Starting point: A maximal degeneration of $(\PP^2,E)$}
\label{Subsect: Maximal degeneration}

For toric surfaces such as $(\PP^2,D)$ with $D$ the union of coordinate lines,
the intrinsic mirror indeed gives the Hori-Vafa mirror
\[
xyz= t^\ell
\]
where $\ell\in \NE(\PP^2)=\NN$ is the class of a line, a generator. If we
replace $D$ by a smooth elliptic curve $E$, the intrinsic mirror has central
fiber $\AA^1$. This is still an interesting case since the intrinsic mirror
comes with a distinguished basis of regular functions (``generalized theta
functions'') and the structure constants carry enumerative information, see
\cite{Gr22a,W}.

To obtain a family of two-dimensional varieties as a mirror, we rather need to
apply the mirror construction to a simple normal crossings degeneration
$g:(\shY,\shD)\to C$ over a curve $C$ with general fiber $(\PP^2,E)$ and $\shD$
with a zero-dimensional stratum. Thus, $\shD$ is a simple normal crossings
divisor in a smooth $\shY$ surjecting to the smooth curve $C$ and with
$\shY_0=g^{-1}(0)\subseteq \shD$ for a special closed point $0\in C$. The
remaining component of $\shD$ surjects onto $C$ and has general fiber an
elliptic curve. To construct $\shY$, we consider the family $\ol\shY$ of
hypersurfaces
\[
XYZ-s(U+f)=0
\]
in the weighted projective space $\PP(1,1,1,3)$. Here $\deg X=\deg Y=\deg Z=1$,
$\deg U=3$, $f\in\CC[X,Y,Z]$ is a general homogeneous polynomial of degree $3$,
and $s$ is the degeneration parameter. The divisor $\ol\shD\subset\ol \shY$ is
defined by $V(sU)$, so is a union of the central fiber $\ol \shY_0$ and an
irreducible divisor restricting to $U=0$ in each fiber.

This family is easy to understand by viewing $\PP(1,1,1,3)$ as the toric variety
with momentum polytope with vertices $(-1,-1,0)$, $(2,-1,0)$, $(-1,2,0)$,
$(0,0,1)$. This is a tetrahedron with one facet $\sigma$ a standard $2$-simplex
scaled by $3$, and the vertex $v$ not in this face has integral distance~$1$
from this face. The monomial $U$ corresponds to $v$, and $X^3,Y^3,Z^3,XYZ$ to
the vertices and only interior integral point in $\sigma$. For $s\neq 0$ we can
eliminate $U$ to obtain $\PP^2$ with the standard homogenous coordinates
$X,Y,Z$. The divisor $U=0$ leads to the family $XYZ-sf(X,Y,Z)=0$ of smooth cubic
curves.
\begin{figure}[h]
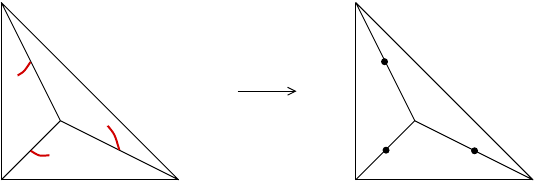
\caption{Central fiber $\shY_0$ of the degeneration $\shY$ of $(\PP^2,E)$}
\label{Fig: Y_0}
\end{figure}

The fiber $\shY_0$ at $s=0$ is sketched in Figure~\ref{Fig: Y_0} on the right.
We have three irreducible components, each a weighted projective plane
$\PP(1,1,3)$, meeting at the point $p=[0,0,0,1]\in \PP(1,1,1,3)$. Since $U+f$ is
not zero at $p$, the total space $\ol\shY$ inherits the toroidal singularity of
$\PP(1,1,1,3)$ at $p$, a $\CC^3/(\ZZ/3)$ quotient singularity, with $\ZZ/3$
acting diagonally by a third root of unity.

The three edges connecting $v$ to $\sigma$ do not have an interior integral
point. This implies that $U+f$ has a simple zero in the interior of each of the
three $\PP^1$'s covering the double locus of $\ol{\shY}_0$. These zeros give three
$A_1$-singularities in $\ol\shY$, locally analytically isomorphic to $uv=s w$
and depicted by the solid circles in Figure~\ref{Fig: Y_0} on the right. These
have small resolutions by locally blowing up one of the two adjacent irreducible
components, see Figure~\ref{Fig: Y_0} on the left for a symmetric choice of
resolution. This procedure replaces the singular points by a $\PP^1$
lying in the blown up irreducible component.

Since the resulting flat family $\shY\to \AA^1$ inherits the toroidal
singularity at $p$ from the ambient $\PP(1,1,1,3)$, it is not simple normal
crossings as required in \cite{GS19,GS22}. We could easily resolve this
singularity torically. However, Johnston in \cite[Rem.\,1.3]{J0} showed that
toroidal singularities work just as well, with an obvious adjustment to the
definition of the tropicalization explained in \S\ref{Subsect: Tropicalization}.
Since the family $\shY$ with the toroidal singularity is also what is discussed
in \cite{CPS,Gr20}, we do not resolve further. We define $\shD\subset\shY$ as
the preimage of $\ol\shD\subset\ol\shY$.

\begin{remark}
\label{Rem: Uniqueness of shY}
From the moduli perspective, one could consider the stack $\scrM$ of flat
families $(\shY,\shD)\to S$ of pairs \'etale locally isomorphic to a smooth
family of plane cubics. This is a smooth Deligne-Mumford stack of dimension one
with coarse moduli space $\AA^1$ and a unique maximal degeneration point.

In fact, the $9$~flex points of a smooth cubic curve are the intersection points
of the remarkable Hesse configuration of $12$~lines. Each of the lines contains
three flex points and each flex point is contained in four lines. The stabilizer
subgroup of $\PGL(3)$ of the Hesse configuration is a finite group $\Gamma$ of
order $216$ acting transitively both on the set of lines and the set of flex
points. Now the embedding of a geometric fiber $E$ of $\shD\to S$ into $\PP^2$
can be fixed by choosing three flex points $p_0,p_1,p_2\in E$ not contained in a
line. Indeed, choosing $p_0$ as the origin of $E$ as an elliptic curve,
$p_1,p_2$ are generators of the $3$-torsion subgroup $E[3]\simeq(\ZZ/3\ZZ)^3$.
There is then a unique embedding $E\to \PP^2$ mapping $p_0,p_1,p_2, p_1+p_2$ to
$[1,0,0],[0,1,0],[0,0,1],[1,1,1]$, respectively.

For a family $(\shY,\shD)\to S$, the set of flex points of the fibers of $\shD\to
S$ are an \'etale cover of $S$ of degree~$9$ which can be assumed trivial after
a finite \'etale cover of $S$. The previous construction then extends to the
whole family to obtain an isomorphism of $(\shY,\shD)$ with a smooth family of
elliptic curves in $\PP^2_S$ with a level~3 structure defined by the chosen
sections $p_0,p_1,p_2$.

Thus $\scrM$ has a finite \'etale cover by the moduli stack of elliptic curves
with (full) level~$3$ structure. The latter moduli space is a scheme whose
universal curve can explicitly be given by the Hesse pencil of plane cubics
\[
\lambda\cdot XYZ+\mu\cdot (X^3+Y^3+Z^3)=0,
\]
minus the four singular members, each a union of three lines, at $\lambda/\mu\in
A=\{0,\omega,\omega^2,1\}$ for $\omega$ a primitive third root of unity. The
base locus of the Hesse pencil is the set of nine flex points of each member, so
the flex points do not depend on $[\lambda,\mu]$. The automorphism group
$\Gamma$ of the Hesse configuration induces an action on the Hesse pencil. The
induced action on the parameter space $\PP^1$ factors over a faithful action of
the alternating group $A_4$ permuting the critical values. See \cite{AD} for a
comprehensive discussion. This shows
\[
\scrM\simeq \big[(\PP^1\setminus A)/\Gamma\big],
\]
and the uniqueness of the maximal degeneration point of $\scrM$.

To make the connection to our family $(\shY,\shD)\to\AA^1$, observe that
after a projective automorphism and rescaling of $s$ we may choose
$f=X^3+Y^3+Z^3$. Then $(\shY,\shD)$ is a partial compactification of the Hesse
pencil of cubics. The previous discussion thus shows that any maximal
degenerating family of plane cubics is locally isomorphic to $(\shY,\shD)$ away
from the central fiber, up to base change.
\end{remark}


\subsection{Dual intersection complexes}
\label{Subsect: Dual intersection complexes}

Recall that a toroidal pair $(Y,D)$ is a variety $Y$ and divisor $D$ that
\'etale locally is isomorphic to a toric variety and its toric boundary divisor.
In both \cite{GS19,GS22}, a central object is the tropicalization
$\Sigma(Y)=\Sigma(Y,D)$ of a toroidal pair, or rather a subcomplex called the
Kontsevich-Soibelman skeleton. The two complexes only differ if $K_Y+D\neq 0$,
so the distinction is irrelevant in the present case and we will only explain
the former. If $Y$ is smooth and $D\subset Y$ is a simple normal crossings
divisors then $\Sigma(Y)$ is the dual intersection cone complex.

The construction of $\Sigma(Y)$ as a conical polyhedral complex with integral
structure for a toroidal pair without self-intersections has been introduced in
\cite[p.71]{KKMS}. It runs as follows. As a matter of notation, we write
$\tau_\ZZ$ for the set of integral points of a rational polyhedral cone $\tau$,
and conversely write $P_\RR$ for the rational polyhedral cone generated by a
finitely generated submonoid $P$ in a lattice. Let $D=D_1+\ldots+D_r$ be the
decomposition into irreducible divisors. For each stratum of $Y$ of
codimension~$k$ along which $(Y,D)$ is locally analytically isomorphic to the
toric variety $\Spec\CC[\tau^\vee_\ZZ]$ with $\tau\subset \RR^n$ a strongly
convex rational polyhedral cone, one adds $\tau$ to $\Sigma(Y)$. An inclusion of
strata induces an inclusion of cones $\tau\to\tau'$, and $\Sigma(Y)$ is the
colimit in the category of integral cone complexes for this diagram of cones.
There is then a one-to-one inclusion-reversing correspondence between cones
$\tau$ in $\Sigma(Y)$ and strata $Y^\circ_\tau$ of $(Y,D)$. We indicate with a
circle superscript that all but the minimal strata are \emph{not} closed in $Y$.
The closure of $Y_\tau^\circ$ is denoted $Y_\tau$ and referred to as a closed
stratum.

Thus $Y$ itself corresponds to the zero-cone, rays in $\Sigma(Y)$ are in
bijection with the irreducible components of $D$, and maximal cones correspond
to the locally minimal strata.

If $D_{i_1},\ldots,D_{i_k}$ are the irreducible components of $D$ containing a
stratum $Y_\tau^\circ$, we can describe $\tau$ intrinsically as follows. Denote
by
\begin{equation}
\label{Eqn: P from divisors}
\textstyle
P_\tau=\big\{\sum_{\mu=1}^k a_\mu D_{i_\mu}\text{ is Cartier \'etale locally
near $Y_\tau^\circ$} \,\big|\, a_\mu \ge0\big\}
\end{equation}
the monoid of effective Cartier divisors near $Y_\tau$ supported on $D$. Then
\[
\tau=\Hom(P_\tau,\RR_{\ge 0}).
\]
This follows immediately since the statement is true for $Y$ the affine toric
variety $\Spec\CC[P]$, $P=\tau^\vee_\ZZ$, and $D$ the complement of the
torus.

If $D=D_1+\ldots+D_r$ is a simple normal crossings divisor, one obtains a simple
global description of $\Sigma(Y)$ as a union of simplicial cones in $\RR^r$:
\[
S=\big\{I\subseteq\{1,\ldots,r\}\,\big|\,
{\textstyle\bigcap_{i\in I} D_i\neq\emptyset\big\}}, \quad
\big|\Sigma(Y)\big|=\bigcup_{I\in S} \sigma_I \subseteq \RR^r.
\]
Here $\sigma_I= \textstyle \sum_{i\in I} \RR_{\ge0}\cdot e_i$ is the cone
spanned by the unit vectors $e_i$ with $i\in I$.

The importance of $\Sigma(Y)$ in the intrinsic mirror constructions from
\cite{GS19,GS22} comes from the fact that its set of integral points
$\Sigma(Y)(\ZZ) =\bigcup_\tau \tau_\ZZ$ is the set of \emph{contact orders with
$D$} of maps from curves $C\to Y$ at points of isolated intersections with $D$.
Indeed, if $f:C\to Y$ is a map from a smooth curve with $f(C)\not\subseteq D$
and $h\in \O_{Y,f(x)}\setminus\{0\}$ fulfills $V(h)\subseteq D$ then
$\ord_x(h)\in \NN$. Thus if $Y_\tau$ is the stratum containing $x$, the discrete
valuation $\nu_x:\O_{C_x}\setminus\{0\}\to \NN$ of $C$ at $x$ composed with
pull-back $f^\sharp$ by $f$ can be evaluated at locally defining functions of
the Cartier divisors supported on $D$ at $f(x)$ to define a map
\[
P_\tau\lra \NN,
\]
that is, an element of $\tau_\ZZ=\Hom (P_\tau,\NN)$.


\subsection{The tropicalization of $(\shY,\shD)$}
\label{Subsect: Tropicalization}
Applying the tropicalization construction to $(\shY,\shD)$ gives three standard
simplicial cones $\sigma_1,\sigma_2,\sigma_3=\RR_{\ge0}^3$ for the three normal
crossings points on $\shY_0$ depicted as the vertices at the boundary of the
large triangles in Figure~\ref{Fig: Y_0}, while a local toric computation at
the remaining zero-dimensional stratum of $\shY_0$ gives
\[
\sigma_0\simeq\RR_{\ge0}\cdot(-1,0,1)+\RR_{\ge0}\cdot(0,-1,1)
+\RR_{\ge0}\cdot(1,1,1),
\]
a non-standard simplicial cone. The projection $\shY\to \AA^1$ tropicalizes to a
map of cone complexes
\begin{equation}
\label{Eqn: Sigma(shY) -> RR+}
g:\big|\Sigma(\shY)\big|\lra \big|\Sigma(\AA^1)\big|=\RR_{\ge0},
\end{equation}
allowing to picture $\Sigma(\shY)$ by the induced polyhedral decomposition $\P$
of $g^{-1}(1)$ (Figure~\ref{Fig: Sigma(shY)}). We write $\ol\sigma_i=
\sigma_i\cap g^{-1}(1)$. Note that $g^{-1}(0)$ is the unique ray common to
$\sigma_1,\sigma_2,\sigma_3$, and it is parallel to the unbounded line segments
in $\ol\sigma_1, \ol\sigma_2, \ol\sigma_3$. This ray corresponds to the
component $D_4$ of $\shD$ surjecting to $\AA^1$, the degenerating family of
elliptic curves, hence is generated by the valuation on local coordinate rings
defined by $D_4$. In the description of $P_{\sigma_i}$ in \eqref{Eqn: P from
divisors}, this valuation maps $\sum_{\mu=1}^k a_\mu D_{i_\mu}$ to the
coefficient of $D_4$.
\begin{figure}[h]
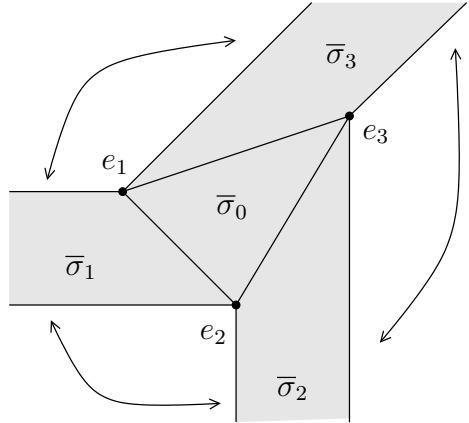
\caption{The fiber $g^{-1}(1)$ of the tropicalization
$\Sigma(\shY)\to\Sigma(\AA^1)$ of $\shY\to\AA^1$ with the induced polyhedral
decomposition $\P$. The arrows indicate identical line segments, the vertices
$e_1,\ldots,e_3$ are integral generators of the rays of $\sigma_0$.}
\label{Fig: Sigma(shY)}
\end{figure}


\subsection{The integral affine manifold $B$ and asymptotic geometry}
\label{Subsect: Affine structure}

In the case of a toric variety, $\Sigma(Y)$ is the set of cones of a fan, hence
is a manifold with an integral affine structure, a system of charts with
transition functions in $\Aff(\ZZ^n)$. If $(Y,D)$ is not toric, but a simple
normal crossings logarithmic Calabi-Yau variety with a zero-dimensional stratum,
there is still an integral affine structure away from the codimension two
skeleton of $\Sigma(Y)$. The reason is that, in this case, each compact
one-dimensional stratum $C\subseteq D$ is isomorphic to $\PP^1$ with exactly two
zero-dimensional strata \cite[Prop.\,1.3]{GS22}. Thus $C\subset Y$ looks like
the inclusion of the closure of a one-dimensional stratum in a toric variety.
Now the fan structure of a toric variety near a one-codimensional cone can be
reconstructed by intersection numbers of the corresponding toric curve with the
toric Cartier divisors. The same formula provides the extension of affine
structure on the interiors of maximal cones in $\Sigma(Y)$ over the codimension
one cones.

Specifically, if $e_1,\ldots,e_{n-1}$ span a codimension one cone $\rho$ in
$\Sigma(Y)$ then $C=Y_\rho\simeq\PP^1$, and the statement on zero-dimensional
strata on $C$ means there are exactly two maximal cones $\sigma,\sigma'$ in
$\Sigma(Y)$ with $\rho=\sigma\cap\sigma'$. Let $e_n, e'_n$ be the primitive
generators of the rays of $\sigma,\sigma'$ not in $\rho$, and denote by
$D_{i_\alpha}$ the component of $D$ corresponding to $e_\alpha$,
$\alpha=1,\ldots, n-1$. Then there is a unique map $\sigma\cup\sigma'\to \RR^n$ linear on each cone with the property
\begin{equation}
\label{Eqn: affine chart equation}
e_n+e'_n= -\sum_{\alpha=1}^{n-1} (D_{i_\alpha}\cdot C)\cdot e_\alpha
\end{equation}
and which maps $e_1,\ldots,e_n$ to the standard unit vectors. This map defines
the affine chart on $|\Sigma(Y)|$ near $\Int\rho$.

Returning to our maximal degeneration $(\shY,\shD)$ of $(\PP^2,E)$, only the
affine structure on $\Sigma(\shY)$ on the cones $\sigma_1,\sigma_2,\sigma_3$
over the unbounded cells in Figure~\ref{Fig: Sigma(shY)} is relevant for this
article. Let $e_1,e_2,e_3$ be the generators of the rays of $\sigma_0$, the
vertices in Figure~\ref{Fig: Sigma(shY)}, and $e_4$ the remaining ray generator
of $\sigma_1,\sigma_2,\sigma_3$ parallel to the unbounded line segments in
$\ol\sigma_i$, $i=1,2,3$. By symmetry, it suffices to consider the affine
structure in the codimension one cone spanned by $e_2,e_4$. The corresponding
curve $C\subset\shY$ is one of the three $\PP^1$'s sketched as an outer edge in
Figure~\ref{Fig: Y_0}. The components $D_2,D_4$ of $\shD$ corresponding to
$e_2,e_4$ are the component of the central fiber $\shY_0$ containing $C$ and the
horizontal divisor $D_4$ already mentioned above. From $D_1+D_2+D_3=\shY_0$, we
obtain
\[
-D_2\cdot C= (D_1+D_3)\cdot C=2,
\]
while using the symmetry and observing that $D_4\cap \shY_0$ is a degeneration
of a cubic curve shows
\[
D_4\cdot C=\frac{1}{3}\cdot (E\cdot E) = 3.
\]
The affine structure near the cone spanned by $e_2,e_4$ according
to \eqref{Eqn: affine chart equation} is therefore given by
\[
e_1+e_3= -(D_2\cdot C)\cdot e_2-(D_4\cdot C) \cdot e_4= 2e_2 -3 e_4
\]
Choosing $e_2=(0,0,1)$, $e_3=(0,1,1)$, $e_4=(1,0,0)$, we obtain
\[
e_1= -(0,1,1)+2\cdot(0,0,1)- 3\cdot(1,0,0)= (-3,-1,1).
\]
The resulting affine geometry on the union of unbounded cells in
Figure~\ref{Fig: Sigma(shY)} is depicted in Figure~\ref{Fig: Asymptotic
geometry}.
\begin{figure}[h]
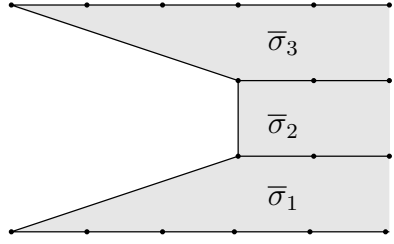
\caption{Asymptotic affine geometry on $B=g^{-1}(1)\subset
\big|\Sigma(\shY)\big|$. The polyhedra $\overline{\sigma}_i$ extend infinitely
to the right. Shown is an affine chart covering $B\setminus \ol\sigma_0$ with
both rays on the top and the bottom mapping onto $\ol\sigma_1\cap\ol\sigma_3$.}
\label{Fig: Asymptotic geometry}
\end{figure}
The vertices are at positions $(-3,-1),(0,0),(0,1),(-3,2)$. The diagram could be
periodically continued vertically to yield a chart on the universal cover of
$B\setminus\ol\sigma_0$ via powers of the affine coordinate transformation
$x\mapsto A\cdot x+b$ with
\begin{equation}
\label{Eqn: Total affine monodromy}
A=\left(\begin{matrix}1&-9\\0&1\end{matrix}\right),\quad
b=\left(\begin{matrix}-9\\3\end{matrix}\right)
\end{equation}
Note that this affine structure does not depend on how we resolved $\ol \shY$.
That choice does, however, influence the affine structure along the facets of
the central cone $\sigma_0$. Figure~\ref{Fig: Sigma(shY)} shows the affine
structure induced by the resolution in Figure~\ref{Fig: Y_0} in a neighborhood
(shaded gray) of the interiors of the three facets of $\ol\sigma_0$. We skip the
computation since only the affine structure on the union of unbounded cells is
relevant to our period computations.

We denote $B=g^{-1}(1)\subset \big|\Sigma(\shY)\big|$ with its integral affine
structure away from the vertices of the polyhedral decomposition $\P$ induced by
$\Sigma(\shY)$, hence with maximal cones $\ol\sigma_0,\ldots,\ol\sigma_3$. Note
that the union of unbounded cells $\ol\sigma_1,\ol\sigma_2,\ol\sigma_3$ of $\P$
has the topology of $S^1\times\RR_{\ge 0}$. The matrix $A$ in \eqref{Eqn: Total
affine monodromy} is the linear part of the affine monodromy given by parallel
transport of integral tangent vectors along the generator of
$\pi_1(S^1\times\RR_{\ge0})$.


\subsection{Wall structures and intrinsic mirror families}
\label{Subsect: Intrinsic mirror construction}

To simplify the following discussion of the intrinsic mirror
construction, we assume $D$ supports a nef divisor. In \cite{GS19}, the
intrinsic mirror ring of $(\shY,\shD)$ is then defined directly as the
free $\CC\lfor P\rfor$-module
\begin{equation}
\label{Eqn: R as a module}
R=\bigoplus_{p\in\Sigma(\shY)(\ZZ)} \CC\lfor P\rfor\cdot \vartheta_p
\end{equation}
with one module generator for each contact order $p$, and multiplication
\begin{equation}
\label{Eqn: multiplication of theta functions}
\vartheta_p\cdot\vartheta_q=\sum_{r\in \Sigma(\shY)(\ZZ)}
\sum_{A\in P} N_{pqr}^A \vartheta_r t^A.
\end{equation}
The structure coefficients $N_{pqr}^A\in\QQ$ are logarithmic Gromov-Witten
counts of genus zero stable log maps with curve class $A$ and contact orders
$p,q$ and $-r$ with $\shD$. The negative contact order $-r$ is implemented by
the theory of punctured logarithmic maps \cite{ACGS}. If $\tau$ is the minimal
cone containing $r$ then the irreducible component containing the marked point
of the domain of a punctured stable map with contact order $-r$ contributing to
$N_{pqr}^A$ maps into the closed stratum $\shY_\tau$. The main result of
\cite{GS19} is that \eqref{Eqn: multiplication of theta functions} defines a
ring structure on $R$, and in particular is associative. For the grading of $R$
one takes $\deg\vartheta_p$ as the integral length of $g(p)$ for
$g:|\Sigma(\shY)|\to \RR_{\ge0}$ as in \eqref{Eqn: Sigma(shY) -> RR+}, the
contact order with the central fiber $\shY_0$.

Better control of intrinsic mirrors, at the expense of slightly more restrictive
assumptions, which hold in all cases relevant to mirror symmetry, is provided by
the alternative construction of $R$ in \cite{GS22}. This approach employs the
previously developed machinery of wall structures and generalized theta
functions \cite{GS11, GPS, GHK, GHS}. A wall $\fop$ in this setup is a
codimension one rational polyhedral subset of $|\Sigma(\shY)|$ contained in one
cone $\sigma$ of $\Sigma(\shY)$ along with an element $f_\fop$ of the Laurent
polynomial ring $\CC\lfor P\rfor [\fop_\ZZ^\gp]$ with exponents integral tangent
vector fields on $\fop$ and coefficients in $\CC\lfor P\rfor$, the completed
ring of curve classes. The construction requires $f_\fop$ to be homogeneous of
degree $0$, so all exponents are contracted by $g|_\sigma$, and $f_\fop\equiv 1$
modulo the maximal ideal $\maxid=(P\setminus\{0\})$ of $\CC\lfor P\rfor$.

Walls are naturally swept out by the endpoints of universal families of
tropicalizations of punctured maps of a given type with only one puncture if the
family of endpoints happens to be one-codimensional. Each type
$\fot$ of such tropical curves (\emph{wall types}) and curve classes $A\in P$
yields a punctured invariant $W_{\fot,A}$ and associated wall $\fop_{\fot,A}$
with
\begin{equation}
\label{Eqn: Canonical wall function}
f_{\fop_{\fot,A}}=\exp\left(k_\fot W_{\fot,A} t^A z^{-u_\fot}\right),
\end{equation}
see \cite[(3.9)--(3.11)]{GS22}. Here $u_\fot\in \sigma_\ZZ\subseteq
\Sigma(\shY)(\ZZ)$ is the contact order at the puncture, $\sigma$ is the
smallest cone containing $\fop$, and $k_\tau\in\NN\setminus\{0\}$ is a tropical
multiplicity. The set of all such walls defines the \emph{canonical wall
structure} of \cite{GS22}. Figure~\ref{fig:univcov} shows the canonical wall
structure of $(\shY,\shD)$ in the asymptotic chart of Figure~\ref{Fig:
Asymptotic geometry} up to curves of degree $3$.
\begin{figure}[h!]
\centering
\begin{tikzpicture}[scale=0.85,define
rgb/.code={\definecolor{mycolor}{RGB}{#1}}, rgb color/.style={define
rgb={#1},mycolor}]]
\begin{scope}
\clip (-5,-3) rectangle (12,4);
\draw[->,color=\myclr] (-9.0, -2.0) -- (-11.25, -2.5);
\draw[->,color=\myclr] (-9.0, 3.0) -- (-11.25, 3.5);
\draw[->,color=\myclr] (-6.0, -1.5) -- (-12.0, -2.5);
\draw[->,color=\myclr] (-6.0, 2.5) -- (-12.0, 3.5);
\draw[->,color=\myclr] (-3.0, -1.0) -- (-5.25, -2.5);
\draw[->,color=\myclr] (-6.0, -2.0) -- (-6.75, -2.5);
\draw[->,color=\myclr] (-1.5, -0.5) -- (-7.5, -2.5);
\draw[->,color=\myclr] (-9.0, -2.0) -- (-10.5, -2.5);
\draw[->,color=\myclr] (-1.5, 1.5) -- (-7.5, 3.5);
\draw[->,color=\myclr] (-9.0, 3.0) -- (-10.5, 3.5);
\draw[->,color=\myclr] (-3.0, 2.0) -- (-5.25, 3.5);
\draw[->,color=\myclr] (-6.0, 3.0) -- (-6.75, 3.5);
\draw[color=\myclr] (0.0, 0.5) -- (0.0, -2);
\draw[->,color=\myclr] (-6.0, -2.0) -- (-6.0, -2.5);
\draw[->,color=\myclr] (-3.0, -1.0) -- (-3.0, -2.5);
\draw[->,color=\myclr] (-9.0, -2.0) -- (-9.0, -2.5);
\draw[color=\myclr] (0.0, 0.5) -- (0.0, 3);
\draw[->,color=\myclr] (-6.0, 3.0) -- (-6.0, 3.5);
\draw[->,color=\myclr] (-3.0, 2.0) -- (-3.0, 3.5);
\draw[->,color=\myclr] (-9.0, 3.0) -- (-9.0, 3.5);
\draw[->,color=\myclr] (0.0, 1.0) -- (18.0, 1.0);
\draw[->,color=\myclr] (-3.0, 2.0) -- (18.0, 2.0);
\draw[->,color=\myclr] (-9.0, 3.0) -- (18.0, 3.0);
\draw[->,color=\myclr] (0.0, 0.0) -- (18.0, 0.0);
\draw[->,color=\myclr] (-3.0, -1.0) -- (18.0, -1.0);
\draw[->,color=\myclr] (-9.0, -2.0) -- (18.0, -2.0);
\draw[->,color=\myclr] (-9.0, -2.0) -- (18.0, -2.0);
\draw[->,color=\myclr] (-9.0, 3.0) -- (18.0, 3.0);
\draw[->,color=\myclr] (-3.0, -1.0) -- (18.0, -1.0);
\draw[->,color=\myclr] (0.0, 1.0) -- (18.0, 1.0);
\draw[->,color=\myclr] (0.0, 0.0) -- (18.0, 0.0);
\draw[->,color=\myclr] (-3.0, 2.0) -- (18.0, 2.0);
\draw[->,color=\myclr] (0.0, 1.5) -- (18.0, 1.5);
\draw[->,color=\myclr] (-4.5, 2.5) -- (18.0, 2.5);
\draw[->,color=\myclr] (1.5, 0.5) -- (18.0, 0.5);
\draw[->,color=\myclr] (0.0, -0.5) -- (18.0, -0.5);
\draw[->,color=\myclr] (-4.5, -1.5) -- (18.0, -1.5);
\draw[->,color=\myclr] (-9.0, 3.0) -- (18.0, 3.0);
\draw[->,color=\myclr] (0.0, 1.0) -- (18.0, 1.0);
\draw[->,color=\myclr] (-3.0, 2.0) -- (18.0, 2.0);
\draw[->,color=\myclr] (-3.0, -1.0) -- (18.0, -1.0);
\draw[->,color=\myclr] (-9.0, -2.0) -- (18.0, -2.0);
\draw[->,color=\myclr] (0.0, 0.0) -- (18.0, 0.0);
\draw[->,color=\myclr] (-6.0, -2.0) -- (18.0, -2.0);
\draw[->,color=\myclr] (0.0, -1.0) -- (18.0, -1.0);
\draw[->,color=\myclr] (3.0, 1.0) -- (18.0, 1.0);
\draw[->,color=\myclr] (3.0, 0.0) -- (18.0, 0.0);
\draw[->,color=\myclr] (-6.0, 3.0) -- (18.0, 3.0);
\draw[->,color=\myclr] (0.0, 2.0) -- (18.0, 2.0);
\draw[->,color=\myclr] (-5.0, -1.667) -- (18.0, -1.667);
\draw[->,color=\myclr] (0.0, -0.667) -- (18.0, -0.667);
\draw[->,color=\myclr] (0.0, 1.667) -- (18.0, 1.667);
\draw[->,color=\myclr] (-5.0, 2.667) -- (18.0, 2.667);
\draw[->,color=\myclr] (1.0, 1.333) -- (18.0, 1.333);
\draw[->,color=\myclr] (1.0, -0.333) -- (18.0, -0.333);
\draw[->,color=\myclr] (2.0, 0.667) -- (18.0, 0.667);
\draw[->,color=\myclr] (-3.0, -1.333) -- (18.0, -1.333);
\draw[->,color=\myclr] (2.0, 0.333) -- (18.0, 0.333);
\draw[->,color=\myclr] (-3.0, 2.333) -- (18.0, 2.333);
\draw[->,color=\myclr] (0.0, 0.0) -- (3.75, -2.5);
\draw[->,color=\myclr] (0.0, -1.0) -- (2.25, -2.5);
\draw[->,color=\myclr] (-1.5, 1.5) -- (10.5, -2.5);
\draw[->,color=\myclr] (0.0, -1.0) -- (4.5, -2.5);
\draw[->,color=\myclr] (0.0, 0.0) -- (7.5, -2.5);
\draw[->,color=\myclr] (-3.0, -1.0) -- (1.5, -2.5);
\draw[->,color=\myclr] (-0.0, -2.0) -- (1.5, -2.5);
\draw[->,color=\myclr] (-6.0, -2.0) -- (-4.5, -2.5);
\draw[->,color=\myclr] (-4.5, -1.5) -- (-1.5, -2.5);
\draw[->,color=\myclr] (-3.0, -2.0) -- (-1.5, -2.5);
\draw[->,color=\myclr] (-1.5, -0.5) -- (10.5, 3.5);
\draw[->,color=\myclr] (0.0, 2.0) -- (4.5, 3.5);
\draw[->,color=\myclr] (0.0, 1.0) -- (7.5, 3.5);
\draw[->,color=\myclr] (-3.0, 2.0) -- (1.5, 3.5);
\draw[->,color=\myclr] (-0.0, 3.0) -- (1.5, 3.5);
\draw[->,color=\myclr] (-6.0, 3.0) -- (-4.5, 3.5);
\draw[->,color=\myclr] (-4.5, 2.5) -- (-1.5, 3.5);
\draw[->,color=\myclr] (-3.0, 3.0) -- (-1.5, 3.5);
\draw[->,color=\myclr] (0.0, 1.0) -- (3.75, 3.5);
\draw[->,color=\myclr] (0.0, 2.0) -- (2.25, 3.5);
\draw[->,color=\myclr] (-6.0, 2.5) -- (18.0, -1.5);
\draw[->,color=\myclr] (3.0, -0.0) -- (18.0, -2.5);
\draw[->,color=\myclr] (0.0, 1.0) -- (18.0, -2.0);
\draw[->,color=\myclr] (0.0, 0.0) -- (15.0, -2.5);
\draw[->,color=\myclr] (6.0, -1.0) -- (15.0, -2.5);
\draw[->,color=\myclr] (0.0, -1.0) -- (9.0, -2.5);
\draw[->,color=\myclr] (0.0, -0.5) -- (12.0, -2.5);
\draw[->,color=\myclr] (3.0, -1.0) -- (12.0, -2.5);
\draw[->,color=\myclr] (-6.0, -1.5) -- (18.0, 2.5);
\draw[->,color=\myclr] (3.0, 1.0) -- (18.0, 3.5);
\draw[->,color=\myclr] (0.0, 0.0) -- (18.0, 3.0);
\draw[->,color=\myclr] (0.0, 1.0) -- (15.0, 3.5);
\draw[->,color=\myclr] (6.0, 2.0) -- (15.0, 3.5);
\draw[->,color=\myclr] (0.0, 2.0) -- (9.0, 3.5);
\draw[->,color=\myclr] (0.0, 1.5) -- (12.0, 3.5);
\draw[->,color=\myclr] (3.0, 2.0) -- (12.0, 3.5);
\draw[->,color=\myclr] (0.0, 1.0) -- (15.75, -2.5);
\draw[->,color=\myclr] (3.0, 0.0) -- (14.25, -2.5);
\draw[->,color=\myclr] (-13.5, 3.5) -- (18.0, 0.0);
\draw[->,color=\myclr] (-3.0, 2.0) -- (18.0, -0.333);
\draw[->,color=\myclr] (0.0, 1.0) -- (18.0, -1.0);
\draw[->,color=\myclr] (9.0, -0.0) -- (18.0, -1.0);
\draw[->,color=\myclr] (3.0, 1.0) -- (18.0, -0.667);
\draw[->,color=\myclr] (3.0, 0.0) -- (18.0, -1.667);
\draw[->,color=\myclr] (1.5, 0.5) -- (18.0, -1.333);
\draw[->,color=\myclr] (6.0, 0.0) -- (18.0, -1.333);
\draw[->,color=\myclr] (-13.5, -2.5) -- (18.0, 1.0);
\draw[->,color=\myclr] (-3.0, -1.0) -- (18.0, 1.333);
\draw[->,color=\myclr] (0.0, 0.0) -- (18.0, 2.0);
\draw[->,color=\myclr] (9.0, 1.0) -- (18.0, 2.0);
\draw[->,color=\myclr] (3.0, 0.0) -- (18.0, 1.667);
\draw[->,color=\myclr] (3.0, 1.0) -- (18.0, 2.667);
\draw[->,color=\myclr] (1.5, 0.5) -- (18.0, 2.333);
\draw[->,color=\myclr] (6.0, 1.0) -- (18.0, 2.333);
\draw[->,color=\myclr] (0.0, 0.0) -- (15.75, 3.5);
\draw[->,color=\myclr] (3.0, 1.0) -- (14.25, 3.5);
\draw[->,color=\myclr] (-9.0, 3.0) -- (18.0, 0.75);
\draw[->,color=\myclr] (-3.0, 2.0) -- (18.0, 0.25);
\draw[->,color=\myclr] (0.0, 2.0) -- (18.0, 0.5);
\draw[->,color=\myclr] (0.0, 1.5) -- (18.0, 0.0);
\draw[->,color=\myclr] (-9.0, -2.0) -- (18.0, 0.25);
\draw[->,color=\myclr] (-3.0, -1.0) -- (18.0, 0.75);
\draw[->,color=\myclr] (0.0, -1.0) -- (18.0, 0.5);
\draw[->,color=\myclr] (0.0, -0.5) -- (18.0, 1.0);
\draw[->,color=\myclr] (-3.0, 2.0) -- (18.0, -0.8);
\draw[->,color=\myclr] (-9.0, 3.0) -- (18.0, 1.2);
\draw[->,color=\myclr] (-6.0, 3.0) -- (18.0, 1.4);
\draw[->,color=\myclr] (-4.5, 2.5) -- (18.0, 1.0);
\draw[->,color=\myclr] (-9.0, -2.0) -- (18.0, -0.2);
\draw[->,color=\myclr] (-6.0, -2.0) -- (18.0, -0.4);
\draw[->,color=\myclr] (-4.5, -1.5) -- (18.0, -0.0);
\draw[->,color=\myclr] (-3.0, -1.0) -- (18.0, 1.8);
\draw[->,color=\myclr] (-9.0, 3.0) -- (18.0, 0.429);
\draw[->,color=\myclr] (-9.0, -2.0) -- (18.0, 0.571);
\draw[->, thick,color=gray] (-9,3) -- (18,3);
\draw[->, thick,color=gray] (-3,2) -- (18,2);
\draw[->, thick,color=gray] (0,1) -- (18,1);
\draw[->, thick,color=gray] (0,0) -- (18,0);
\draw[->, thick,color=gray] (-3,-1) -- (18,-1);
\draw[->, thick,color=gray] (-9,-2) -- (18,-2);
\draw[-, thick,color=gray] (0,0) -- (0,1);
\draw[-, thick,color=gray] (-3,2) -- (0,1);
\draw[-, thick,color=gray] (-3,-1) -- (0,0);
\draw[-, thick,color=gray] (-3,2) -- (-9,3);
\draw[-, thick,color=gray] (-3,-1) -- (-9,-2);
\draw[->,ultra thick,densely dotted,color=blue] (1.5, 0.5) -- (18.0, 0.5);
\draw[ultra thick,densely dotted,color=blue] (0,1) -- (1.5, 0.5) -- (0,0);
\fill[color=blue] (0,0) circle (.06cm);
\fill[color=blue] (0,1) circle (.06cm);
\fill[color=blue] (1.5,.5) circle (.06cm);
\end{scope}
\end{tikzpicture}
\caption{Tropical curves in the chart of Figure~\ref{Fig: Asymptotic geometry}
up to degree 3. In dotted blue a maximally tangent tropical conic contributing
to an unbounded wall. This image was produced by \protect\cite{Gr21}.}
\label{fig:univcov}
\end{figure}
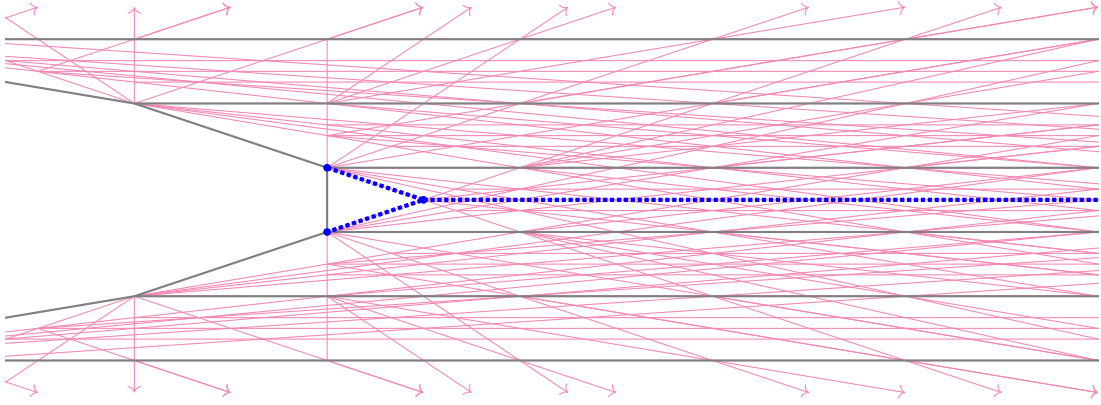

One main result of \cite{GS22} is that the canonical wall structure is
\emph{consistent}, which means that $B$ gives rise to a compatible directed
system of schemes \cite{GHS}. We quickly sketch this construction. We work over
an Artinian quotient $S_I=\CC[P]/I$, $I\subseteq\maxid$ to reduce to finitely
many non-trivial walls. In any case, a wall $\fop$ intersecting the interior of
a maximal cell $\sigma$ induces an $S_I$-algebra isomorphism
\begin{equation}
\label{Eqn: theta_fop}
\theta_\fop: S_I[\sigma^\gp_\ZZ]\lra S_I[\sigma^\gp_\ZZ]
\end{equation}
by splitting $\sigma_\ZZ^\gp =\fop_\ZZ^\gp \oplus\ZZ\cdot \xi$ and defining
\begin{equation}
\label{Eqn: theta_fop on monomials}
\theta_\fop(z^\xi)= f_\fop\cdot z^\xi,\quad
\theta_\fop(z^m)=z^m\text{ for }m\in\fop^\gp_\ZZ.
\end{equation}
The union of the finitely many non-trivial walls subdivide the maximal cell
$\sigma$ into a number of connected components. A \emph{chamber} of the wall
structure is the closure of such a connected component. Taking one copy
$R_\fou=S_I[\sigma_\ZZ^\gp]$ for each chamber $\fou$, and the wall crossing
automorphism $\theta_\fop$ as a morphism $R_\fou\to R_{\fou'}$ for chambers
$\fou,\fou'$ separated by $\fop$, with $\xi$ pointing from $\fou'$ to $\fou$, we
obtain a diagram of rings related by isomorphims. Consistency in codimension
zero says that for each $\fou$ the projection
\[\textstyle
\liminv_{\fou'\subseteq\sigma} R_{\fou'} \lra R_\fou=S_I[\sigma_\ZZ^\gp]
\]
is an isomorphism. In other words, for chambers $\fou,\fou'\subseteq \sigma$ any
chain of wall crossing automorphisms connecting the rings $R_\fou$ and
$R_{\fou'}$ leads to the same isomorphism $R_\fou\to R_{\fou'}$.

If a wall lies in a one-codimensional cone $\rho$ of $\Sigma(\shY)$, which is
then called \emph{a slab} and denoted with the symbol $\fob$ rather than $\fop$
for distinction, one defines a ring
\begin{equation}
\label{Eqn: R_fob}
R^I_\fob=S_I[\rho_\ZZ^\gp][u,v]\big/\big(uv-f_\fob\cdot t^{[\shY_\rho]}\big).
\end{equation}
Here we use that the closed stratum $\shY_\rho$ is one-dimensional, hence has an
associated class $[\shY_\rho]\in P$. If $\sigma,\sigma'$ are the maximal cones
in $\Sigma(\shY)$ containing $\rho$ let $\xi\in (\sigma-\rho)_\ZZ$ be contracted
by $g$ and generate a complementary summand to $\rho_\ZZ^\gp$ in
$\sigma_\ZZ^\gp$. Mapping $u$ to $z^\xi$ and $v$ to $f_\fob t^{[\shY_\rho]}
z^{-\xi}$ induces an isomorphism of the localization
\begin{equation}
\label{Eqn: R_fob -> Laurent}
(R^I_\fob)_u\simeq S_I[\sigma_\ZZ^\gp].
\end{equation}
For the analogous isomorphism involving $\sigma'$, we use $\xi'=-\xi$ via an
affine chart for $\Sigma(\shY)$ containing $\Int\rho$ and map $u, v$ to $f_\fob
t^{[\shY_\rho]} z^{-\xi'}$, $z^{\xi'}$, respectively, to obtain
\begin{equation}
\label{Eqn: R_fob -> Laurent, other side}
(R^I_\fob)_v\simeq S_I[{\sigma'}_\ZZ^\gp].
\end{equation}
Observe that different intersection numbers in the equation \eqref{Eqn: affine
chart equation} defining the affine structure have the effect of multiplying
$f_\fob$ by a monomial. So the affine structure in codimension one is a key
ingredient in the construction.

Taken together one obtains a system of rings with elements
$R_\fou=S_I[\sigma_\ZZ^\gp]$ for chambers $\fou$ in maximal cones $\sigma$ of
$\Sigma(\shY)$, and $R^I_\fob$ for walls $\fob$ contained in codimension one
cells, with arrows the isomorphisms $\theta_\fop$ from \eqref{Eqn: theta_fop}
and localization maps from \eqref{Eqn: R_fob -> Laurent}. Consistency at the
lowest level (``in codimensions zero and one'') says that the cocyle condition
is fulfilled to obtain a scheme $X_I^\circ$ over $\Spec S_I$. This scheme
$X_I^\circ$ is a flat deformation of the complement $X_0^\circ$ of the
codimension two strata of $X_0$ constructed as $\Proj$ of the Stanley-Reisner
ring mentioned at the beginning of \S\ref{sec:degeneration}, a reducible scheme.

To extend this flat deformation to all of $X_0$ requires in addition
consistency of the wall structure in codimension two \cite[Def.\,3.2.1]{GHS}. This
condition amounts to saying that the restriction of $X_I^\circ$ to an affine
subset $U\subseteq X_0$ has enough regular functions to induce a flat
deformation $U_I$ of $U$. One then obtains a flat deformation of a scheme that
is projective over an affine scheme with central fiber $X_0$, and a
distinguished $S_I$-module basis $\vartheta_p^I$ of its homogeneous coordinate
ring $R_I$ for an ample line bundle coming with the construction. There is one
generator $\vartheta^I_p\in R_I$ for each contact order $p\in\Sigma(\shY)(\ZZ)$,
and these are compatible with changing $I$.

Another key result of \cite{GS22} is that $R_I$ agrees with the reduction modulo
$I$ of the intrinsic mirror ring $R$, for all $I$:
\[
R\equiv R_I\quad \mod I.
\]
This isomorphism maps the abstract module generator $\vartheta_p$ from
\eqref{Eqn: R as a module} to $\vartheta^I_p$ for all $I$.


\subsection{Geometry of the intrinsic mirror $\shX\to\Spec\CC\lfor P\rfor$ of
$(\shY,\shD)$}
\label{Subsect: Geometry of intrinsic mirror}

Let $R$ be the intrinsic mirror ring of our degenerating family $(\shY,\shD)$ of
smooth cubic curves embedded in $\PP^2$ and
\[
q:\shX=\Proj R\lra \Spec\CC\lfor P\rfor
\]
the corresponding mirror family. Looking at the integral generators of the cones
in $\Sigma(\shY)$, we see that $R$ is generated as a $\CC \lfor P\rfor$-algebra
by the five $\vartheta_p$ with $p$ the vertices $e_1,e_2,e_3$ of the bounded
cell $\ol\sigma_0\subset B$, the common ray generator $e_4$ of
$\sigma_1,\sigma_2,\sigma_3$, and the interior integral point $e_0$ of
$\sigma_0$. Writing $\vartheta_i=\vartheta_{e_i}$ we have
\[
\deg\vartheta_4=0,\quad \deg\vartheta_i=1\text{ for }i=0,1,2,3.
\]
The closed fiber $\shX_0=X_0=\Proj R/\maxid R$ of $q$ is defined by the
(generalized) Stanley-Reisner ring for $\P$,
\begin{equation}
\label{Eqn: Stanley-Reisner ring of shY}
R/\maxid R= \CC[\vartheta_0,\ldots,\vartheta_4]/
(\vartheta_1\vartheta_2\vartheta_3-\vartheta_0^3,\,\vartheta_4\vartheta_0),
\end{equation}
see \cite[\S2.1]{GHS}. Thus $\shX_0$ has irreducible components (I)~the
hypersurface $X_{\sigma_0}\simeq V(XYZ-U^3)\subset \PP^3$, the toric variety
with momentum polytope $\ol\sigma_0$ and three $A_2$-singularities, and
(II)~three copies of $X_{\sigma_i}=\PP^1\times \AA^1$, one for each of the
remaining maximal polytopes $\ol\sigma_i$, $i=1,2,3$ of $\P$ in Figure~\ref{Fig:
Sigma(shY)}.

Analyzing curves of low degree and using the uniqueness of consistent wall
structures \cite[Prop.\,4.1]{GS11}, one can show that up to a trivial
change\footnote{The only difference is that the singularities of the affine
structure on $B$ are moved from the interior of the bounded edges to the
vertices, see \cite{Gr20}.} the canonical wall structure agrees with the
algorithmically constructed wall structure from \cite{GS11,CPS}, see
\cite{Gr22a}. For the sequel, we only need the following statement.

\begin{lemma}
\label{Lem: Intrinsic mirror is algorithmic mirror}
All walls $\fop$ of the canonical wall structure of $(\shY,\shD)$ are disjoint
from $\Int\sigma_0$. Moreover, if $\delta:B\to \RR_{\ge0}$ denotes the integral
affine distance function from $\ol\sigma_0$ then for any outgoing contact order
$u_\fot$ of a wall type $\fot$ it holds
\[
\nabla_{u_\fot}\delta\ge 0
\]
for the directional derivative, with strict inequality iff the wall is not
contained in $\partial\sigma_0$.
\end{lemma}

\begin{proof}
The generalized balancing condition for a punctured stable map $f:C\to \shY$
over a log point \cite[Cor.\,2.30]{ACGS} implies
\[
\deg f^*\big(\O_{\shY}(-D_4)\big) = -\nabla_{u_\fot}\delta.
\]
The statement now follows by noting that $D_4$ is semi-ample, with $D_4\cdot
C=0$ for an irreducible curve $C\subset \shY_0$ iff $C$ is one of the
exceptional divisors, the curves indicated by little red arcs in
Figure~\ref{Fig: Y_0}.
\end{proof}

The theta function $\vartheta_p$ is defined as an element of the ring
$R_\fou=S_I[\sigma_\ZZ^\gp]$ for a chamber $\fou$ in a maximal cone $\sigma$ in
$\Sigma(\shY)$ by a sum over so-called \emph{broken lines} with initial
direction $-p$ and ending at a general, specified point $x\in\Int\fou$. A broken
line $\beta$ is a piecewise straight path in $|\Sigma(\shY)|$, contained in the
complement of the codimension two skeleton and carrying a monomial $c\cdot z^m$
on each straight line segment, with $m$ tangent to $\beta$ and pointing
backwards. So $m=p$ on the unique unbounded line segment of $\beta$. When
$\beta$ meets a wall $\fop$, one extends $\beta$ across the wall with one of the
summands in the expansion of $\theta_\fop(c z^m)$ into a sum of Laurent
monomials. We refer to \cite[\S3.1]{GHS} for the precise definition. In
particular, the exponent $m$ carried by $\beta$, hence the direction of $\beta$,
can only change at a wall by adding a multiple of one of the exponents appearing
in $f_\fop$. The contribution of each broken line to $\vartheta_p$ for the local
expression in $S_I[\sigma_\ZZ^\gp]$ is the monomial $c z^m$ at its endpoint
$x\in\Int\sigma$.

From the outward pointing nature of our wall structure (Lemma~\ref{Lem:
Intrinsic mirror is algorithmic mirror}), it is then easy to see that the only
broken lines contributing to $\vartheta_0,\ldots,\vartheta_3$ at a point close
to the barycentric ray of $\sigma_0$ are straight (no bend), while those
contributing to $\vartheta_4$ bend at most once when crossing from an unbounded
chamber into $\sigma_0$, by a primitive integral tangent vector of the crossed
facet of $\sigma_0$. The first relation in \eqref{Eqn: Stanley-Reisner ring of
shY} therefore holds to all orders, while the relation involving $\vartheta_4$
for the chosen symmetric resolution $\shY\to\ol\shY$ turns out to give the
Hori-Vafa mirror superpotential \cite[Cor.\,7.9]{CPS}.

\begin{proposition}
\label{Prop: Intrinsic mirror}
The intrinsic mirror of $(\shY,\shD)$, after the base change $\CC\lfor P\rfor\to
\CC\lfor t\rfor$ given by $[C]\mapsto [C]\cdot D_4$, equals
\[
\shX=\Proj \CC\lfor t\rfor[X,Y,Z,U,W]\big/\big(XYZ-U^3,WU-t\cdot(X+Y+Z)\big).
\]
Here $X=\vartheta_1$, $Y=\vartheta_2$, $Z=\vartheta_3$, $U=\vartheta_0$ are of
degree~$1$ and $W=\vartheta_4$ is of degree~$0$.
\end{proposition}

It is also not hard to see that the result does not depend on the choice of
resolution $\shY\to\ol\shY$.

\begin{remark}
\label{Rem: Base change E->0}
The natural base ring for the intrinsic mirror of $(\shY,\shD)$ is $\CC\lfor
P\rfor$ with $P=\NE(\shY)$. This differs from our ring by replacing the
$t$-coefficient in front of $X,Y,Z$ in the second relation by factors $t^{E_i}$
or $t^{E_i+E_j}$, depending on the choice of resolution $\shY\to\ol\shY$. The
base-change to $\CC\lfor t\rfor$ maps all these factors to $1$. The additional
parameters trivialized under this base change are irrelevant for the enumerative interpretation of the period
integral and are therefore disregarded.
\end{remark}

An important insight for our period computation is that the intrinsic mirror
ring has an additional $\ZZ$-grading by putting
\[
\deg X=\deg Y=\deg Z=\deg U=0,\quad \deg t=\deg W=1.
\]
This grading defines a $\GG_m$-action on $\shX\to \Spec\CC\lfor t\rfor$.

Such an action always exists on intrinsic mirrors of degenerations of Fano
manifolds with smooth anticanonical divisor:

\begin{proposition}
\label{Prop: GG_m-action}
Let $(\shY,\shD)\to S$ be a normal crossings degeneration over the germ of a
smooth curve $(S,0)$ with $\shD=\shY_0\cup\shE$ and $\shE$ relatively
ample, irreducible and anticanonical. Denote by $R/\CC\lfor t\rfor$ the
intrinsic mirror ring of $(\shY,\shD)$ from \cite{GS22} for the curve class map
$\NE(\shY)\to P=\NN$ given by pairing with $\shD$, by $B$ the associated
integral affine manifold and by $W$ the theta function defined by $\shE$.

Then there is a $\ZZ$-grading on $R$ with $\deg\vartheta_p=0$ for all
$p\in B(\ZZ)$ and
\[
\deg t= \deg W=1.
\]

\end{proposition}

\begin{proof}
A natural $\ZZ$-grading on the homogenous coordinate ring defined by a wall
structure arises when the wall structure is homogeneous in the sense of
\cite[Thm.\,4.4.3]{GHS}. For a wall structure defined over $\CC\lfor Q\rfor$,
this homogeneity involves the definition of two homomorphisms of abelian groups
\[
\delta_Q: Q\lra \ZZ,\quad \delta_B: \PL(B)^*\lra \ZZ
\]
fulfilling a compatibility condition \cite[(4.8)]{GHS}. Here $\PL(B)$ is the
monoid of continuous maps $B\to \RR$ that are piecewise affine with integral
slopes on each cell of the polyhedral decomposition $\P$ of $B$. Explicitly,
denoting by $\Sigma(\shY)[1]$ the set of rays of $\Sigma(\shY)$, the dual
$\PL(B)^*$ is the quotient of $\ZZ^{\Sigma(\shY)[1]}$ by the subspace of tuples
pairing trivially with all piecewise linear functions. In our case of $Q=P=\NN$
we take $\delta_Q=\id$ and $\delta_B$ the evaluation at the PL-function $\delta$
from Lemma~\ref{Lem: Intrinsic mirror is algorithmic mirror} that vanishes on
all bounded cells of the polyhedral decomposition $\P$ of $B$ and has slope $1$
on the unbounded rays.

It remains to check that the wall functions from \eqref{Eqn: Canonical wall
function} are homogeneous of degree~$0$ for the degree of monomials defined from
$\delta_Q$ and $\delta_B$. Each wall function is a function in a monomial
$t^Az^{-u_\tau}$ for a curve class $A$ and outgoing contact order $u_\tau$ for a
type $\tau$ of a $1$-punctured map. In the notation of \cite{GHS}, this monomial
has exponent $m=(\ol m,m_Q)$ with $\ol m= -u_\tau$, $m_Q=A\in P=\NN$, and degree
\[
\deg(m) = \delta_Q(m_Q) +\delta_B(\nabla_{\ol m}).
\]
We have $\delta_Q(m_Q)=A\in\NN$, the intersection number $[C]\cdot\shE$ for the
class of a curve $C$ contributing to the wall. In the second summand,
$\nabla_{\ol m}\in\PL(B)^*$ is the directional derivative along $\ol m=
-u_\tau$, so $\delta_B(\nabla_{\ol m}) =\nabla_{\ol m}(\varphi)$. The balancing
condition \cite[Cor.\,2.30]{ACGS} applied with $s$ the section of $\ol\M_X^\gp$
defined by $\shE$ then shows
\[
\delta_B(\nabla_{\ol m})= -\nabla_{u_\tau}(\delta)= -[C]\cdot\shE.
\]
Hence $\deg(m)=0$ as required.
\end{proof}


\section{Period integrals in the intrinsic mirror family}
\label{sec:period}


\subsection{Analytification of the intrinsic mirror family}
\label{Subsect: Analytification}
A priori, $\shX$ is only a scheme of finite type over $\Spec\CC\lfor t\rfor$.
But the relations in Proposition~\ref{Prop: Intrinsic mirror} are indeed
polynomial. Note that polynomiality holds whenever $\shD$ supports an
ample divisor, as discussed at the beginning of \S\ref{sec:degeneration}.
Thus $\shX$ is the base change to $\Spec\CC\lfor t\rfor$ of a flat
scheme over $\AA^1$. Denote by
\[
\pi:\shX_\an \lra \CC,
\]
its analytification, the set of solutions of
\begin{equation}
\label{Eqn: equations for mirror}
XYZ=U^3,\ WU=t\cdot(X+Y+Z)
\end{equation}
in $\PP^3\times\CC^2$ with homogeneous coordinates $X,Y,Z,U$ for $\PP^3$ and
$W,t$ the standard coordinates on $\CC^2$.

Note that for a scheme of finite type over $\CC \lfor P\rfor$, such as $\Proj$ of
the mirror ring in general, one can always obtain an analytic approximation to
order $k$ by cutting off coefficients in the $(k+1)$-th power of the maximal
ideal $\maxid= \big(P\setminus\{0\}\big)\subset \CC\lfor P\rfor$. Such an
approximation would be sufficient for the following computation.

An important non-algebraic function for our computations is the monomial $w$ for
the outgoing primitive tangent vector $(1,0)$ in the asymptotic chart
Figure~\ref{Eqn: Total affine monodromy}, that is, the generator $e_4$ of the
ray $\sigma_1\cap\sigma_2\cap\sigma_3$. Recall from \S\ref{Subsect: Intrinsic
mirror construction} that the intrinsic mirror ring $R_I$ modulo $I\subset
\CC\lfor P\rfor$ was given as a ring of functions on a scheme covered by affine
open subschemes of the form $\Spec S_I[\sigma_\ZZ^\gp]$ for $\sigma$ a maximal
cone in $\Sigma(\shY)$ and $\Spec R^I_\fob$ for $\fob$ a wall contained in a
codimension one cone. For $\sigma$, $\fob$ intersecting the asymptotic chart
Figure~\ref{Eqn: Total affine monodromy}, all of these rings contain the
distinguished monomial $w=z^{(1,0)}$. Of course, $w$ is only invariant under
wall crossing automorphisms for walls $\fob$ containing $(1,0)$ in their tangent
space. We claim that nevertheless $w$ has a meaning on a large region of $X_I$, compatibly for all $I$.

For any $I\subseteq\maxid$ with $S_I=\CC[P]/I$ Artinian, there exists $k\in\NN$
with $I\subseteq kP$. Hence $t^p\in S_I$, $p\in P$, can be non-zero only for the
$p$ in the finite set $P\setminus kP$. This implies that for the computation of
$R_I$ only finitely many walls matter, and in turn that there is a cocompact
region in $\sigma_1\cup \sigma_2\cup \sigma_3$ only containing walls parallel to
$(1,0)$. The Zariski-open subset of $X_I=\Proj R_I$ covered by the spectra of
these asymptotic model rings is the complement of the complete irreducible
component $X_{\sigma_0}\subset |\shX_I|=\shX_0$. Thus $\shX_I$ contains a regular
function $w$ that restricts to $z^{(1,0)}$ on each of the model affine open
subsets.

But note that $w$ does not, or at least not obviously, lie in the intrinsic
mirror ring $R$, a finitely generated $\CC\lfor P\rfor$-algebra. In the case of
our degeneration $(\shY,\shD)$ of $(\PP^2,E)$, we will rather express $w^3/t^3$
in \S\ref{Subsect: Canonical coordinates} as an exponentiated integral of the
holomorphic $2$-form $\Omega_{\shX_\an}$ over a family of $2$-chains with
boundaries on fibers of $W$. This description breaks down, and in fact exhibits
multi-valued behavior, near the set of critical values $3\mu_3$, $\mu_3=\{1,e^{\pm 2\pi i/3}\}$. Thus there is at least a holomorphic function on
\[
\big\{ x\in \shX_\an\,\big|\, |W(x)|>3|\pi(x)|\big\}
\]
restricting to $w$ on the analytification of $\shX_I\setminus X_{\sigma_0}$ for
all $I$. In the following, we view $w$ as this holomorphic function. A similar
argument works in great generality.

Alternatively, one could invert the expansion
of $W$ from \cite[Thm.\,5.12]{CPS},
\begin{equation}
\label{Eqn: W=W(w,t)}
W=w+\sum_\ell \ell N_{1,\ell} w^{-\ell}t^{\ell+1} =
w\cdot\Big(1+\sum_\ell \ell N_{1,\ell}(t/w)^{\ell+1}\Big),
\end{equation}
over $\CC\lfor t\rfor$ and truncate to express $w$ as a holomorphic function in
$W$ and $t$ in an open set of $\shX_\an$ containing $\shX_0 \setminus
X_{\sigma_0}$, up to terms of order $t^{k+1}$. The coefficients $N_{1,\ell}$ are
logarithmic Gromov-Witten invariants with two marked points of contact orders
$1$ and $\ell$ with $\shE$, with the first contact point with $\shE$ specified
\cite[Thm.\,4.14]{GS22}. Note that the polarization by $\shE$ implies that any
summand in \eqref{Eqn: W=W(w,t)} is a constant multiple of $w^{-\ell}
t^{\ell+1}$. Since the polarizing divisor $\shE$ for $(\PP^2,E)$ is
$3$-divisible, we obtain the further restriction $\ell+1=3d$ for some $d$. See
\cite{Gr22a} for a detailed discussion of $W$ in the case of $(\PP^2,E)$.

The bracketed expression on the right-hand side of \eqref{Eqn: W=W(w,t)}
curiously is closely related to the mirror map of local $\PP^2$
\cite[Thm\,1.1]{GRZ}. See also \cite[Prop.\,4.1]{GRZZ} for an explicit
expression of the power series expansion of the inverse $w(W)$.

In the rest of this section we work in the analytic category and omit the
subscript ``$\an$''. Thus from now on, $\shX$ is a complex analytic space with
two holomorphic functions $\pi,W:\shX\to \CC$. The additional holomorphic
function $w$ is defined on an open subset of $\shX$ containing $\shX_0\setminus
X_{\sigma_0}$ and is a function in $W$ and $t$. We assume $W$ and $w$ are real
for the given real structure on $\shX$ and agree with the analytifications of
the regular functions denoted by the same symbols on $\shX_I$ and
$\shX_I\setminus X_{\sigma_0}$ for $I=(t^{k+1})$ and some fixed $k\gg 0$. These
properties are automatic in our case, and in general can be achieved by an
appropriate choice of analytic approximations. We are going to check identities
between holomorphic functions up to holomorphic multiples of $t^{k+1}$ and then
take the limit $k\to\infty$.

The construction described in \S\ref{Subsect: Intrinsic mirror construction}
also shows that the logarithmic canonical bundle of $\pi\colon \shX\to\CC$ is
trivial. Indeed, this follows by the theorem on formal functions and GAGA since
the intrinsic mirror construction comes with a distinguished section $\Omega$ of
the relative logarithmic canonical bundle. On an affine chart $\Spec
S_I[\sigma_\ZZ^\gp]$,
\[
S_I[\sigma_\ZZ^\gp] \cong \CC[t,z_1,z_2]/(t^{k+1}),
\]
we have $\Omega=\dlog z_1\wedge \dlog z_2$. Thus $\Omega$ is the reduction
modulo $I$ of the holomorphic family of holomorphic $2$-forms $\Omega_t$ on
$\shX$ with logarithmic poles along $\shE$ uniquely determined by the
normalization condition
\begin{equation}
\label{Eqn: normalization}
\Pi_0(t):=\frac{1}{(2\pi i)^2}\int_{F_t} \Omega_t = 1.
\end{equation}
Here $F_t$ is an SYZ-fiber in $\shX_t$, a family of $2$-cycles homologous in
$\shX_t$ to $|z_1|=|z_2|=\const$.


\subsection{Positive real locus}
\label{Subsect: positive real locus}

Intrinsic mirror rings $R$ by definition have coefficients in $\QQ$, and in
particular are defined over $\RR$. Thus intrinsic mirrors $\shX=\Proj R$ come
with a real locus $\shX^\RR$. The restriction to the central fiber $\shX_0$ is
the union of real loci of the toric irreducible components, hence a $2^n$-fold
cover of the union of momentum polyhedra $B=\bigcup_i \ol\sigma_i$. After
restricting $\shX\to \CC$ to an interval $(-\epsilon,\epsilon)\subset \CC$,
we may assume that the intersection with $\shX_0$ induces a bijection of
connected components of $\shX^\RR$ with the connected components of
$\shX_0^\RR$.

In our case, as in any case admitting a toric model in the sense of \cite{GHK},
the slab functions $f_\fob$ even have coefficients in $\NN$. We only need this
statement for the lowest $t$-order in each slab function, where it follows by a
direct computation. In fact, any stable map with class $A$ fulfilling $A\cdot
D_4=0$ is a multiple cover of one of the three exceptional curves in $\shY_0$;
these are known to lead to the three slabs covering the facets $\rho_i$ of
$\sigma_0$, with functions $1+ z^m$ and $m$ a primitive integral tangent vector
spanning the tangent space of $\ol\rho_i$.

The positivity of order zero slab functions implies that $\shX_0^\RR =
\shX_0\cap\shX^\RR$ has a distinguished connected component $\shX_0^>$ that on
each irreducible toric component $X_{\sigma_i}$ restricts to the positive real
locus, the closure of $\RR_{>0}^2\subseteq (\CC^*)^2$ in $X_{\sigma_i}$. Denote
by $\shX_t^>\subset \shX_t$ the intersection of the connected component of
$\shX^\RR$ containing $\shX_0^>$ with $\shX_t$.

\begin{lemma}
\label{Lem: Lefschetz-thimbles for W}
For each $t>0$ we have
\[
W\big(\shX_t^>\big)= [3t,\infty),
\]
with fibers homeomorphic to $S^1$ over $s\in(3t,\infty)$ and a point over $s=3t$.
\end{lemma}

\begin{proof}
Dehomogenizing \eqref{Eqn: equations for mirror} by $U$ shows that a dense set
of $\shX_t$ is isomorphic to $(\CC^*)^2$ with coordinates $x,y$, and
$W=t\cdot\big(x+y+1/(xy)\big)$. From this description the statement is immediate
by direct computation already done in \S\ref{Subsect: Takahashi's mirror
family}. The critical points of $W$ are at $x=y$, $x^3=1$, and hence $3t$ is the
only real critical value of $W$.
\end{proof}

The lemma shows that
\[
W^{-1}\big([3t,s]\big)\cap \shX_t^> =
W^{-1}\big((-\infty,s)\big)\cap \shX_t^>
\]
for $s\in (3t,\infty)$ is a Lefschetz-thimble for $W$ intersecting the elliptic
curve $W^{-1}(s)\cap \shX_t$ in its positive real locus, an $S^1$. Since $w$ is
a function of $W$ and $t$, and $W$ and $w$ agree on $\shX_0\setminus
X_{\sigma_0}$, we can also use the $w$-coordinate to parametrize this family of
Lefschetz thimbles:
\begin{equation}
\label{Eqn: Gamma(s,t)}
\Gamma_{s,t}:= w^{-1}\big((-\infty,s)\big)\cap \shX_t^>.
\end{equation}
We will see in \S\ref{Subsect: Canonical coordinates} that $w$ is a canonical
coordinate in this context. The parametrization of the Lefschetz thimbles in
terms of $w$ thus removes a mirror map from our statements.
The definition of $\Gamma_{s,t}$ makes sense for $t/s$ sufficiently small, and
in particular for $s=1$ and $t$ sufficiently small.

To understand the relation of $\Gamma_{s,t}$ with the integral affine manifold
$B$, recall that the momentum maps $X_{\sigma_i}\to\sigma_i$ defined by the
Fubini-Study metric for the monomial basis of sections of $\O(1)$ patch together
to define a degenerate momentum map
\begin{equation}
\label{Eqn: momentum map}
\mu: \shX_0\to B,
\end{equation}
see \cite[Prop.\,1.1]{RS}\footnote{The result in loc.cit.\ assumes $B$ bounded,
but it extends to the general case by adding $\sum_i |\vartheta_{m_i}|^2\cdot
m_i$ in the definition of $\mu_\sigma$ with $m_i$ generators of the monoid of
integral points in the recession cone of $\sigma$. In the present case this
means adding $|W|^2\cdot e_4$ on the unbounded cells.}. Since the positive real
locus of a polarized toric variety maps homeomorphically onto its momentum
polytope, $\mu$ induces a homeomorphism $\shX_0^>\to B$. Now by the explicit
description of $w$ in the affine charts,
\begin{equation}
\label{Eqn: mu(|w|=const)}
\mu(\partial \Gamma_{s,0}) = \mu(|w|=s\big)
\end{equation}
is a circle of three line segments parallel to the edges of $\sigma_0$, cf.\
Figure~\ref{Fig: betatrop} below. Thus $\mu$ induces a homeomorphism of
$\Gamma_{s,0}$ with the closed disk in $B_0$ enclosed by the union of these
three line segments.

The homeomorphism of $\shX^>$ with $B$ extends to small $t>0$ away from a
neighborhood $U\subset B$ of a codimension two locus in $B$, as one can see by
working with the Kato-Nakayama space of $\shX_0$ as a log space
\cite[\S7.1]{KNreal}. Moreover, away from a neighborhood of a codimension two
locus, the full real locus $\shX_t^\RR$ for $t>0$ small is an unbranched cover
of $B\setminus U$. The codimension two locus is the image under $\mu$ of the
zero locus of the slab functions, the center points of the edges of $\sigma_0$
in the case of the mirror of $(\shY,\shD)$. Under the assumption of the
positivity of the lowest order coefficients of the slab functions, $\shX_0^>$ is
disjoint from this zero locus. A local computation near this codimension two
locus shows that the homeomorphism $\shX_0^> \to B$ extends to a family of
homeomorphisms $\shX_t^>\to B$. In fact, the arguments of \cite{KNreal} using
the Kato-Nakayama space show that near points of $\shX_0$ away from the zero
locus of the slab function, the lift of the real locus to the Kato-Nakayama
space maps to $B$ by a local homeomorphism, and this local homeomorphism extends
to small $t$. Alternatively, one could argue by an explicit computation in local
holomorphic coordinates compatible with the real structure. In any case,
$\shX_t^>$ is homeomorphic to $\RR^2$ for all $t\ge 0$, and $\Gamma_{s,t}$ is
topologically a family of disks.

We orient $\Gamma_{s,t}$ by the orientation of $B$ displayed in Figure~\ref{Fig:
betatrop}. Note there is no canonical choice of orientation of $B$ as it depends
on a cyclic ordering of the three irreducible components of $\shX_0$ to determine the orientation of $\ol\sigma_0$.


\subsection{Integral over a tropical 1-chain and canonical coordinates}
\label{Subsect: Canonical coordinates}
We will need another family of 2-chains $\beta_{s,t}$ for small $t>0$, each
homeomorphic to a pair of pants fibering over a graph $\beta^{\text{trop}}_s$ in
$B$. Its construction is modeled after the construction of tropical 1-cycles in
\cite{RS}. The graph $\beta^{\text{trop}}_s$ features one trivalent vertex,
three bivalent and three univalent vertices. The univalent vertices
$v_1,v_2,v_3$ lie on the circle $\mu\big(|w|=s\big)$ from \eqref{Eqn:
mu(|w|=const)}, the trivalent vertex $v_0$ is contained in the central cell of
$B$, and the bivalent vertices lie on the edges of $\ol\sigma_0$ as
sketched.\footnote{The figure is simplified by showing the affine structure
compatible with $\mu$ at the central fiber of the full intrinsic mirror family.
The base change from Remark~\ref{Rem: Base change E->0} has the effect of moving
the singularities of the affine structure at the three vertices to the centers
of the edges of $\ol\sigma_0$, and $\beta_s^\trop$ may have to be deformed
accordingly.}
\begin{figure}[h]
\includegraphics[width=.4\textwidth]{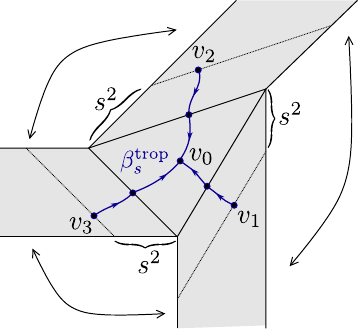}
\caption{The tropical 1-cycle $\beta^{\text{trop}}_s$.}
\label{Fig: betatrop}
\end{figure}
Along the path in $\beta_s^\trop$ connecting a one-valent vertex $v_i$ to
$v_0$, we endow $\beta_{s,t}$ with the parallel transport of the primitive
outward pointing vector $m=(1,0)$ in the asymptotic affine chart in
Figure~\ref{Fig: Asymptotic geometry}. The balancing condition at $v_0$ is
\[
(1,1)+(-1,0)+(0,-1)=(0,0).
\]
Thus $\beta_{s,t}$ with the edges oriented toward $v_0$ is a tropical
$1$-chain in the sense of \cite{RS}, that is, a singular $1$-chain on the
complement of the set of vertices of $B$ with values in the sheaf $\Lambda$ of
integral tangent vectors.

The conormal construction of \cite[\S2.3]{RS} now provides a singular $2$-chain
$\beta_{s,0}$ in $\shX_0$ with fibers $S^1$ over the interior of the edges of
$\beta_s^\trop$ that contract to points at $\partial\sigma_0$, boundary three
copies of $S^1$ in $w^{-1}(s)\cap \shX_0$ and with a triangle inserted over
$v_0$ to form a pair of pants. The property $\partial\beta_{s,t}\subset
w^{-1}(s)$ holds because the edges of $\beta_s^\trop$ in the unbounded cells
carry the asymptotic monomial.

Note that $w^{-1}(s)\cap \shX_0$ is a union of three copies of $\PP^1$ forming a
cycle; indeed, the restriction of $\pi:\shX\to\CC$ to $w=\const$ and $t$
sufficiently small is a base-changed Tate-curve, as is obvious from
\cite[Cor.\,5.13]{CPS}. Each of the three connected components of
$\partial\beta_{s,0}$ splits one of the $\PP^1\subset w^{-1}(s)\cap\shX_0$ into
two connected components.

\begin{figure}
\includegraphics[width=.3\textwidth]{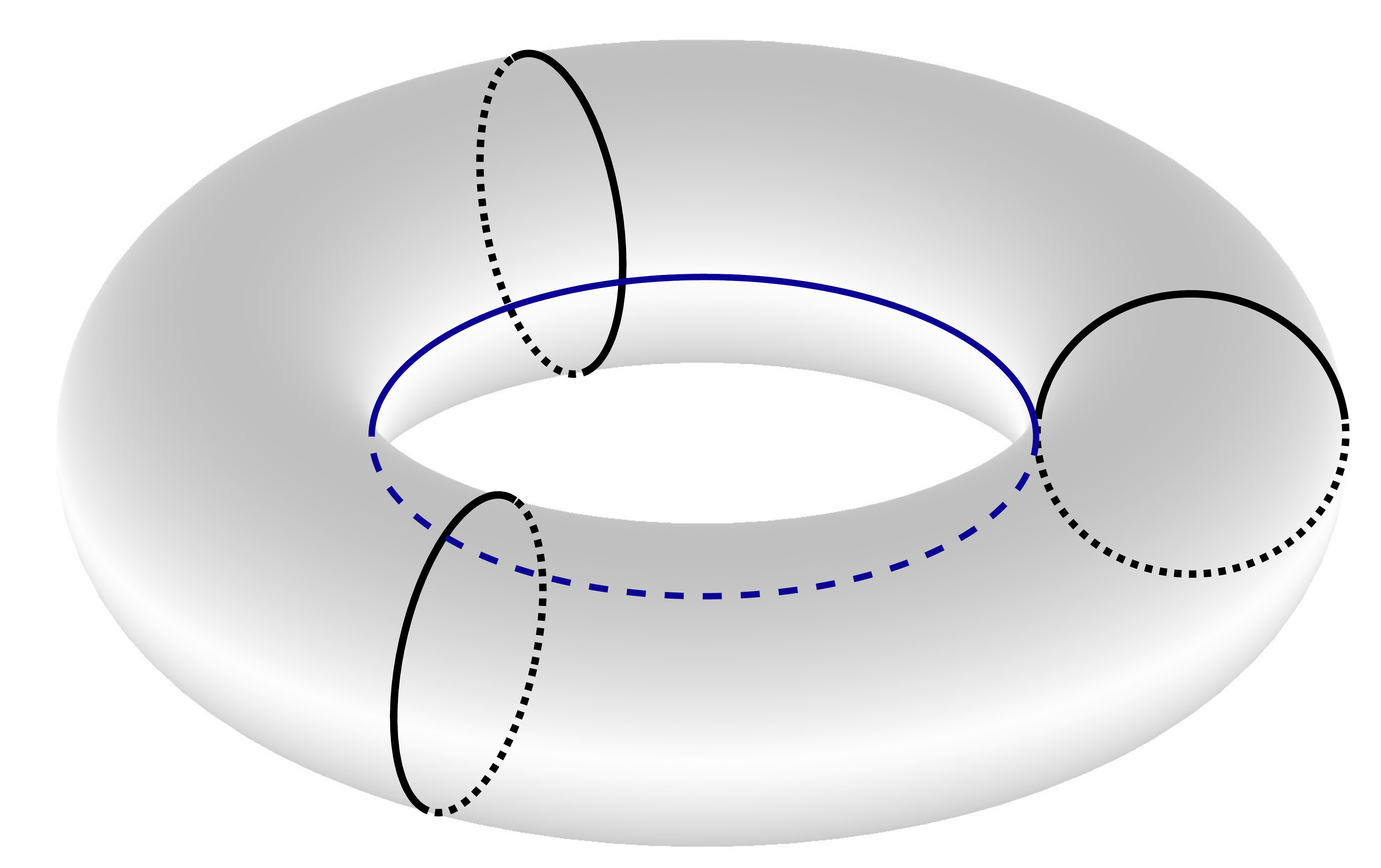}
\caption{Sketch of $\partial\beta_{s,t}$ and $\partial\Gamma_{s,t}$ as three
meridians and an equator in the elliptic curve $w^{-1}(s)\cap\shX_t$.}
\label{torus-fig}
\end{figure}

The construction in \cite{RS} also shows how $\beta_{s,0}$ extends to a
continuous family $\beta_{s,t}$ of $2$-chains in $\shX_t$, for $t$ in a
contractible neighborhood in $\CC^*$ of an interval $(0,t)$ with $t>0$ small.
To treat the boundary, the construction allows to add the property
\[
\partial\beta_{s,t}\subset w^{-1}(s).
\]
The construction of $\beta_{s,t}$ can be extended to any contractible open set
of $(s,t)\in\CC^*\times\CC$ such that $w^{-1}(s)\cap \shX_t$ is not singular if
$t\neq 0$, that is, with $\big(W(s,t)\big)^3\neq (3t)^3$. Figure~\ref{torus-fig}
shows $\partial\beta_{s,t}$ and $\partial\Gamma_{s,t}$ as curves inside the
elliptic curve $w^{-1}(s)\cap\shX_t$.

The period integral $\Pi_1$ of $\Omega_t$ over $\beta_{s,t}$ can be readily
computed by \cite{RS}:

\begin{proposition}
\label{prop:period1}
For any $s,t>0$ with $\beta_{s,t}$ defined we have
\begin{equation}
\label{Eqn: Pi_1(s,t)}
\Pi_1(s,t)=\frac{1}{2\pi i}\int_{\beta_{s,t}} \Omega_t
=  \pi i + 3(\log t - \log s).
\end{equation}
\end{proposition}

\begin{proof}
The integral follows from the computation in \S3.6 of \cite{RS}, and notably
Equation~(3.17) with both the Ronkin function $\shR$, the gluing data
$s_{\sigma\!\!,\;\ul\rho}$ and $s_{\sigma'\!\!\!\!,\;\ul\rho}$ trivial. The integral thus becomes a
sum over simple contributions from the vertices of $\beta^{\text{trop}}_s$. Up
to the global factor ${2\pi i}$, the trivalent vertex contributes $\frac12$,
each bivalent vertex contributes $\log t$ and each univalent vertex contributes
$-\log s$.
\end{proof}

The usual definition of a canonical coordinate now yields,
\[
\exp\left(\frac{\Pi_1(s,t)}{\Pi_0(t)}\right)= \exp\big(\Pi_1(s,t)\big)
=-t^3/s^3.
\]
The first equality follows from the normalization property \eqref{Eqn:
normalization}. Restricting to the distinguished fiber $s=1$ exhibits $-t^3$ as
a canonical coordinate. But for fixed $t>0$, one could equally well view
$s^{-3}$ as a canonical coordinate in the codomain of $W$.

As we will see in the proof of Proposition~\ref{Prop: Takahashi's
conjecture via intrinsic periods}, the occurrence of the first summand $\pi i$
in $\Pi_1(s,t)$, which originates from the trivalent vertex of
$\beta^\trop_{s,t}$, is our explanation for the negative sign in front of $q$ in
\eqref{eq:takItwo}.

A similar construction can be done in general. The resulting period integral
gives different integral coefficients in Proposition~\ref{prop:period1}, but due
to the $\CC^*$-action on the $\shX$ acting with the same weights on $w$ and $t$
(Proposition~\ref{Prop: GG_m-action}) is always an affine function in
$\log(t)-\log(s)$.


\subsection{Main computation: Integration over positive real Lefschetz-thimbles}
\label{sec:main-result}

Our main result is the computation of the period integral
\begin{equation}
\label{Eqn: Pi_2(s,t)}
\Pi_2(s,t)=\int_{\Gamma_{s,t}}\Omega_t
\end{equation}
over the Lefschetz thimbles $\Gamma_{s,t}$ from \eqref{Eqn: Gamma(s,t)} for
cases such as the mirror family $\shX$ of $(\shY,\shD)$. We defined
$\Gamma_{s,t}$ as a subset of the positive real locus $\shX_t^>$, for
$s,t\in\RR_{>0}$ and $t/s$ small enough to fulfill $W(s,t)>3t$, the real
critical value of $W$. Similarly to the family of $2$-chains $\beta_{s,t}$ in
\S\ref{Subsect: Canonical coordinates}, there exists an extension of
$\Gamma_{s,t}$ as a continuous family of $2$-chains not only to positive real
$s,t$ with $t/s$ small, but to any contractible subset of $(s,t)\in (\CC^*)^2$
not containing a critical value of $W$ in the $(w,t)$-coordinates. In our
computations, we nevertheless restrict to $s,t>0$ real and then argue by unique
holomorphic continuation.

Note that for fixed $s$ the holomorphic continuation of the period integral
$\Pi_2(s,t)$ about $t=0$ is multi-valued due to the log poles of $\Omega$ near the
double locus of $\shX_0$. Indeed, such period integrals, for $s$ fixed, may also
involve constant multiples of $\log^2 t$ and products of $\log t$ with a
holomorphic function in $t$, as we explicitly saw in \eqref{eq:periods} in our
review of Takahashi's result.

Our computation works for any two-dimensional $\shX$ obtained via \cite{GHS}
from a real wall structure on an asymptotically cylindrical integral affine
manifold $B$ with integral polyhedral decomposition $\P$. The asymptotically
cylindrical condition means that all unbounded edges in $\P$ are parallel
\cite[Def.\,2.1]{CPS}. Thus such a $B$ has asymptotic charts similar to the one
shown in Figure~\ref{Fig: Asymptotic geometry}. We also assume that all
monomials in the slab and wall functions of unbounded walls are outgoing,
meaning they are polynomials in $w^{-1}$, for $w$ as above the monomial with
exponent the primitive outward pointing integral tangent vector of an unbounded
edge. This assumption is automatically fulfilled by the wall structures from
\cite{GS22} or \cite{GS11}.

The complement of a compact subset $K\subset B$ is homeomorphic to
$\RR_{>0}\times S^1$, with polyhedral decomposition induced by a polyhedral
decomposition of $S^1$ and the affine structure determined by two integers.
These are the integral affine circumference $\ell\in\NN\setminus\{0\}$ and $e\in
\ZZ$ from the linear part of the affine monodromy $\left(\begin{smallmatrix}
1&e\\0&1\end{smallmatrix}\right)$. We can thus write
\begin{equation}
\label{eq-iso-torus}
\textstyle
B\setminus K\simeq \big(\RR_{>0}\times [0,\ell]\big)\big/ \sim,
\end{equation}
where the equivalence relation identifies $\RR_{>0}\times\{0\}$ and
$\RR_{>0}\times\{\ell\}$, and the inclusion of $\RR_{>0}\times(0,\ell)$ is an
affine isomorphism onto the image. In the case of our degeneration of
$(\PP^2,E)$ we have $\ell=3$, $e=9$.

Non-horizontal walls in this asymptotic chart are bounded, see
Figure~\ref{fig:univcov}. By working only with unbounded chambers to construct
the $k$-th order smoothing $\shX_I\setminus X_{\sigma_0}$ of $\shX_0\setminus
X_{\sigma_0}$, $I=(t^{k+1})$, we can thus restrict to unbounded, horizontal
walls. For $k$ fixed we furthermore only consider the finite set of unbounded
walls $\fop$ and slabs $\fop$ that are non-zero modulo $t^{k+1}$, as explained
in \S\ref{Subsect: Intrinsic mirror construction}. This finite wall-structure
with all walls unbounded yields the same $\shX_I\setminus X_{\sigma_0}$, which
in turn agrees with the reduction modulo $t^{k+1}$ of the analytic family
$\shX\to \CC$, restricted to the complement of $X_{\sigma_0}$. We label the
slabs
\[
\fob_0,\ldots,\fob_m
\]
cyclically from bottom to top in a diagram with the asymptotic direction
horizontal to the right.

Denote further by
\begin{equation}
\label{Eqn: kink}
\textstyle
\kappa=\sum_\rho \kappa_\rho\in P=\NN
\end{equation}
the total kink of the
multivalued piecewise linear function $\varphi$ entering the construction of
\cite{GHS} at infinity. In the canonical wall structure, $\kappa_\rho$ is the
curve class associated to $\rho$. In our case we have $\kappa_\rho=3$
for all unbounded $\rho$, and $\kappa=3\cdot 3=9$.

The analytic function $w$ in \S\ref{Subsect: Analytification} restricts to the
unique monomial in
\begin{equation}
\label{Eqn: S_I[..]}
S_I[\sigma_\ZZ^\gp]\simeq\CC[t]/(t^{k+1})[w^{\pm1},x^{\pm1}]
\end{equation}
of degree~$0$ and associated tangent vector $(1,0)$.
Noting that each wall function of an unbounded wall is a Laurent polynomial in
$w^{-1}$ with coefficients in $\CC[t]/(t^{k+1})$, we can then write
\begin{equation}
\label{Eqn: f_out}
f_\out=\prod_\fop f_\fop \cdot\prod_\fob f_\fob \in\RR[t,w^{-1}]
\end{equation}
with all $t$-exponents at most $k$ and $f_\out\equiv 1$ modulo $t$.

Since all these wall and slab functions have a non-zero constant coefficient,
they have no zeros on any interval $[s_0,\infty)$, $s_0>0$, as long as $t>0$ is
sufficiently small. For $s_1>s_0>0$ and such $t$ we now define
\[
\Gamma_{t}(s_0,s_1)= \shX^>\cap w^{-1}\big([s_0,s_1]).
\]
Arguing with the degenerate momentum map $\mu: \shX_0\to B$ as in
\S\ref{Subsect: positive real locus}, it is then not hard to see that
$\Gamma_t(s_0,s_1)$ is a continuous family of cylinders. Moreover, letting
$\delta$ be the integral affine distance function on $B$ from the union of
bounded cells, $\mu$ identifies $\Gamma_0(s_0,s_1)$ with the annulus
\[
\delta^{-1}\big([s_0^2,s_1^2]\big)
\]
in $B$. Indeed, by our definition of $\mu$, the value of $|w|$ on $\mu^{-1}(a)$
equals $a^{1/2}$. Note also that if there exists a positive real Lefschetz
thimble $\Gamma_{s,t}$, as in the case of the mirror of $(\shY,\shD)$ from
\S\ref{Subsect: Maximal degeneration}, then
\begin{equation}
\label{Eqn: Lefschetz cylinder}
\Gamma_t(s_0,s_1)=\Gamma_{s_1,t}-\Gamma_{s_0,t}
\end{equation}
as a singular chain.

After requiring the monomial for $w$ to be given by $(1,0)$, the asymptotic
affine chart induced from \eqref{eq-iso-torus} becomes unique up to an affine
transformation with an integral translation and linear part
$\left(\begin{smallmatrix} 1 & k \\ 0& 1 \end{smallmatrix}\right)$ for
$k\in\ZZ$. We fix one such choice together with a choice of slab $\fob_0$ to
function as a radial cut. Parallel transport of the basis vector $(0,1)$ on the
contractible set $B\setminus\fob_0$ now provides us with a choice of monomials
$u,v$ in each model ring $R_{\fob_i}^I$ from \eqref{Eqn: R_fob}. The $u,v$
generize to monomials with opposite tangent vectors in $S_I[\sigma^\gp_\ZZ]$,
which for $i\neq 0$ we can take to be $x,x^{-1}$ in \eqref{Eqn: S_I[..]}. For
$i=0$ we take $u=x$ and then $v=x^{-1}w^{-e}$ due to the affine monodromy. Note
also that these $(\Spec R_{\fob_i})_\an$ cover $\shX_I\setminus X_{\sigma_0}$
since each other type of model ring $S_I[\sigma_\ZZ^\gp]$ is obtained by a
sequence of localizations \eqref{Eqn: R_fob -> Laurent}, \eqref{Eqn: R_fob ->
Laurent, other side} and isomorphisms \eqref{Eqn: theta_fop}.

Let $u_i$ be an analytic extension of $u\in R_{\fob_i}$ to an open subset
$U_i\subset\shX$, compatible with the real structure and such that
\[
\Omega_t=\dlog w\wedge \dlog u_i
\]
holds locally. Since $v=f_{\fob_i}\cdot t^{\kappa_i}/u$ in $R_{\fob_i}$, the
same formula with $u$ replaced by $u_i$ defines an analytic function $v_i$ on
$U_i$ with
\begin{equation}
\label{Eqn: u_iv_i}
u_iv_i= f_{\fob_i}(w,t)\cdot t^{\kappa_i}.
\end{equation}

We are now ready for the main period computation.

\begin{proposition}
\label{prop:period}
For $s_1>s_0>0$, and $t>0$ sufficiently small we have
\[
\int_{\Gamma_t(s_0,s_1)} \Omega_t = - \int_{s_0}^{s_1}
\Big(\log t^\kappa -\log s^e + \log f_\out(s,t)\Big)\dlog s +O(t^{k+1}).
\]
\end{proposition}

\begin{proof}
The integral is real analytic in $s_0,s_1$. By analytic continuation it
therefore suffices to prove the result for $s_0,s_1\gg0$.

Since the $U_i$ cover $\shX_0\setminus X_{\sigma_0}$, the $U_i$ cover
$\Gamma_t(s_0,s_1)$ for sufficiently small $t>0$. Moreover, $(w,u_i)$ or
$(w,v_i)$ from above provide real, oriented coordinate systems on
$\Gamma_t(s_0,s_1)\cap U_i$. By \eqref{Eqn: u_iv_i} we can also impose the
condition $v_i\le 1$, or equivalently
\begin{equation}
\label{Eqn: u_i>..}
u_i\ge f_{\fob_i}(w,t)\cdot t^{\kappa_i}
\end{equation}
on $U_i$ and still cover $\Gamma_t(s_0,s_1)$ for small $t$.

Noting that the $U_i\cap U_{i+1}$ with $U_{m+1}=U_0$ cover all but the double
locus of $\shX_0\setminus X_{\sigma_0}$, the $U_i\cap U_{i+1}$ also cover
$\Gamma_t(s_0,s_1)$ for $t$ sufficiently small. Denote by $\mathfrak P_i$ the
set of walls between the slabs $\fob_i$ and $\fob_{i+1}$. Then on $U_i\cap
U_{i+1}$, $i=0,\ldots,m-1$, the wall crossing automorphisms \eqref{Eqn:
theta_fop on monomials} provide the relation
\begin{equation}
\label{Eqn: u_iv_i+1}
\textstyle
u_i\cdot v_{i+1}= f_i(w,t)^{-1},\quad
f_i\equiv \prod_{\fop\in\mathfrak P_i}f_\fop\, \mod t^{k+1}.
\end{equation}
The inverse comes from the fact that the exponents for $u_i$ and $v_{i+1}$ both
point into the walls rather than away, in contrast to $u_i,v_i$ in the slab
relation \eqref{Eqn: u_iv_i}. Thus $v_{i+1}\ge 1$ if and only if
\begin{equation}
\label{Eqn: u_i<..}
u_i\le f_i(w,t)^{-1}.
\end{equation}
For $i=m$ the monodromy brings in an additional factor $w^{-e}$ and \eqref{Eqn:
u_iv_i+1} holds with $v_{m+1}=v_0$ and
\begin{equation}
\label{Eqn: u_mv_0}
\textstyle
f_0\equiv w^{-e}\cdot\prod_{\fop\in\mathfrak P_0}f_\fop\, \mod t^{k+1}.
\end{equation}
The inequalities \eqref{Eqn: u_i>..}, \eqref{Eqn: u_i<..} now provide a
decomposition of $\Gamma_t(s_0,s_1)$ for $t>0$ small into the domains
\[
s_0\le w\le s_1,\quad f_{\fob_i}(w,t)\cdot t^{\kappa_i}\le u_i\le f_i(w,t)^{-1},
\]
a subset of $U_i$, $i=0,\ldots,m$.

We obtain
\begin{equation}
\label{Eqn: main integral}
\begin{aligned}
\int_{\Gamma_t(s_0,s_1)} \Omega_t\
&= \sum_{i=0}^m \int_{s_0}^{s_1} \left(\int_{f_{\fob_i}(w,t)\cdot
t^{\kappa_i}}^{f_i^{-1}(w,t)} \dlog u_i \right)\dlog w\\
&= -\sum_{i=0}^m \int_{s_0}^{s_1} \Big(\log(t^{\kappa_i})+
\log f_{\fob_i}(w,t) +
\log f_i(w,t) \Big)\dlog w\\
&= -\int_{s_0}^{s_1}\Big[ \log t^\kappa -  \log w^e +
\log f_\out(w,t)\Big]\dlog w +O(t^{k+1}).
\end{aligned}
\end{equation}
The last equality uses $\kappa=\sum_i\kappa_i$ from \eqref{Eqn: kink} and the
definitions of $f_\out$ and $f_i$ in \eqref{Eqn: f_out}, \eqref{Eqn:
u_iv_i+1}, \eqref{Eqn: u_mv_0}.
\end{proof}

\begin{remark}
For an alternative proof of Proposition \ref{prop:period}, we could have used
\cite[Constr.\,5.5]{CPS} of an asymptotic polyhedral affine pseudomanifold with
asymptotic wall structure that includes walls with reduction modulo $t$ a
constant different from $1$, along with \cite{RS}. The computation then still
reduces to an integral over a degenerating family of elliptic curves.
\end{remark}

We are now ready to prove our main result. Recall that by \cite[(3.11)]{GS22} or
\cite[Thm.\,2]{Gr20}, the infinite product of all asymptotic wall functions is
given by
\begin{equation}
\label{eq:fout}
\log f_{\out}(w,t) = \sum_{d=1}^{\infty} \, 3d \, N_d \, w^{-3d} \, t^{3d}.
\end{equation}

\begin{theorem}
\label{Thm: main period}
Let $\shX\to \CC$ be the analytic intrinsic mirror of our maximal degeneration
$(\shY,\shD)\to\AA^1$ of $(\PP^2,E)$ and $\Gamma_{s,t}$ the positive real
Lefschetz thimble from \eqref{Eqn: Gamma(s,t)} with boundary on $w=s$. Then
\[
t\partial_t\,\Pi_2(s,t)= t\partial_t \int_{\Gamma_{s,t}}\Omega =
9\log (t/s)+\sum_{d\ge 1} 3d N_d (t/s)^{3d}.
\]
\end{theorem}

\begin{proof}
The $\CC^*$-action on $\shX$ of Proposition~\ref{Prop: GG_m-action} implies that
for all $\lambda\in\CC^*$,
\[
\Gamma_{\lambda s,\lambda t}=\lambda\cdot \Gamma_{s,t}.
\]
Here the multiplication by $\lambda$ on the right-hand side denotes the
action. Since $\lambda^*\Omega_{\lambda t}=\Omega_t$, we obtain
\[
\Pi_2(\lambda s,\lambda t)=\int_{\Gamma_{\lambda s,\lambda t}}\Omega_{\lambda t}=\int_{\Gamma_{s,t}}\Omega_t =\Pi_2(s,t).
\]
Taking the derivative with respect to $\lambda$ at $\lambda=1$ thus shows
\[
t\partial_t \int_{\Gamma_{s,t}}\Omega = - s\partial_s \int_{\Gamma_{s,t}}\Omega.
\]
Now for $s_0$ with $0<s_0<s$ and $t>0$ sufficiently small as in Proposition~\ref{prop:period}, we can decompose $\Gamma_{s,t}=\Gamma_{s_0,t}+\Gamma_t(s_0,s)$ as singular chains. Since the integral of $\Omega$ over $\Gamma_{s_0,t}$ does not vary with $s$, we conclude
\[
s\partial_s \,\Pi_2(s,t)= s\partial_s \int_{\Gamma_t(s_0,s)} \Omega
\]
for $s_0<s$ sufficiently close to $s$. The right-hand side can readily be
computed from Proposition~\ref{prop:period} while taking the limit $k\to\infty$
to give
\[
t\partial_t\,\Pi_2(s,t)= -s\partial_s \int_{\Gamma_t(s_0,s)}\Omega
= \log t^\kappa-\log s^e +\log f_\out(s,t).
\]
Noting that $\kappa=e=9$ and plugging in \eqref{eq:fout} gives the stated
formula.
\end{proof}

\noindent
\textbf{Proof of Theorem~\ref{thm:main}.}
The statement follows from Theorem~\ref{Thm:
main period} by setting $s=1$ and integration.
\qed

\begin{remark}
In higher dimensions, an analogous period computation can be done for the
intrinsic mirror of a normal crossings degeneration $(\shY,\shD)$ of a Fano
manifold $Y$ with smooth anticanonical divisor $D$. The positive real locus in
that case has to be replaced by a real one-parameter family of cycles induced by
a tropical $1$-cycle as in \cite{RS} in the asymptotic singular affine manifold
$B_\infty$ from \cite[Constr.\,5.5]{CPS}. At least in the case that this
asymptotic tropical $1$-cycle is the tropicalization of a family of curves
$\shC\subset \shD$, we expect the period integral computes the generating
function of logarithmic Gromov-Witten invariants in $Y$ intersecting $D$ only in
a point of $C=\shC\cap Y$.
\end{remark}


\subsection{Corollary: Takahashi's log mirror symmetry conjecture}
\label{subsec:match}

To deduce Takahashi's enumerative mirror conjecture \eqref{eq:takItwo} from our
periods, note that dehomogenizing the equations for $\shX$ in
Proposition~\ref{Prop: Intrinsic mirror} with respect to $U$ and setting
$x=X/U$, $y=Y/U$ leads to $W=t\cdot\big(x+y+\frac{1}{xy}\big)$. Thus $W=1$
agrees with the family $\check E_t^\circ$ in \eqref{eq:ecirchat} considered by
Takahashi. Next observe that in the canonical coordinate $q=-e^{I_1(Q)}$,
Takahashi's basis $I_0(Q)$, $I_1(Q)$, $I_2(Q)$ of solutions of the Picard-Fuchs
equation is the unique tuple of solutions obeying
\begin{align*}
I_0(q)&= 1\\
I_1(q)&= \log q\\
I_2(q)&= \frac{1}{2}\log^2 q+I_2^\hol(q).
\end{align*}

\begin{proposition}
\label{Prop: Takahashi's conjecture via intrinsic periods}
The periods $\Pi_1(t),\Pi_1(s,t),\Pi_2(s,t)$ from \eqref{Eqn: normalization},
\eqref{Eqn: Pi_1(s,t)}, \eqref{Eqn: Pi_2(s,t)} at $s=1$ fulfill
\begin{align*}
I_0(-t^3)&=\Pi_0(t)=1\\
I_1(-t^3)&=\Pi_1(1,t)=\pi i+3\log t\\
I_2(-t^3)&=\Pi_2(1,t)+c.
\end{align*}
for some $c\in\RR$.
\end{proposition}

\begin{proof}
The first equality is \eqref{Eqn: normalization}. The second equality follows
from Proposition~\ref{prop:period1} by
\[
\exp\big(\Pi_1(1,t)\big)= \exp\big(\pi i+3\log t\big) =-t^3
\]
and the definition of $q$. Finally, Theorem~\ref{thm:main} shows that
$\Pi_2(1,t)$ and $I_2(-t^3)$ are both solutions of the Picard-Fuchs equation
\eqref{eq:pf} with the same coefficient of $\log^2(t)$, namely $1/2$. Since the
space of solutions with non-zero $\log^2(t)$-term and vanishing $\log(t)$- and
constant terms is one-dimensional, we obtain the third equality.
\end{proof}

\begin{corollary}
Takahashi's enumerative mirror conjecture \eqref{eq:takItwo} holds up to an
additive constant.
\end{corollary}

\begin{proof}
The statement follows readily from Proposition~\ref{Prop: Takahashi's conjecture
via intrinsic periods} and Theorem~\ref{thm:main}.
\end{proof}


\sloppy

\end{document}